\documentclass[11pt]{amsart}
\usepackage[margin=1in]{geometry}
\usepackage{xcolor}
\usepackage[english]{babel}

\newcommand{\vp}{\Phi}
 \newcommand{\irnn}{\int_{\Omega_{n}(t)}}

\newcommand{\lz}{L_N}
\newcommand{\mq}{M_q}
\newcommand{\fp}{f^\prime(r)}
\newcommand{\fr}{f(r)}
 \numberwithin{equation}{section}

\newcommand{\ra}{\rightarrow}
\newcommand{\ve}{\varepsilon}

\newcommand{\pt}{\partial_t}

\newcommand{\RN}{\mathbb{R}^N}

\newcommand{\irn}{{\int_{\RN}}}

\newcommand{\rn}{\mathbb{R}^N}

\newcommand{\irnt}{\int_{ Q_T}}

\newcommand{\pxj}{\partial_{x_j}}

\newcommand{\uo}{u^{(0)}}

\newcommand{\iqt}{\int_{Q_T}}
\newcommand{\vpi}{\varphi_i}
\newcommand{\wnr}{\|w\|_{r,Q_T}}
\newtheorem{theorem}{Theorem}[section]

\newtheorem{lemma}[theorem]{Lemma}
\newtheorem{clm}[theorem]{Claim}

\theoremstyle{definition}

\title[the initial value problem for the Nernst-Planck-Navier-Stokes system
] 
      { Global existence of a strong solution to the initial value problem for the Nernst-Planck-Navier-Stokes system in high space dimensions}

\author[Xiangsheng Xu]{}
 \email{xxu@math.msstate.edu}

\begin{document}

\maketitle

\centerline{\scshape Xiangsheng Xu}
\medskip
{\footnotesize
 \centerline{Department of Mathematics \& Statistics}
   \centerline{Mississippi State University}
   \centerline{ Mississippi State, MS 39762, USA}
} 

	\begin{abstract}We study the existence of a strong solution to the initial value problem for the Nernst-Planck-Navier-Stokes (NPNS) system  in $\mathbb{R}^N, N\geq 3$. 
		The system describes the electrodiffusion of ions in a
		viscous Newtonian fluid.  A strong solution is obtained in any dimension of space without constraints on the number of species or the size of the given data.

\end{abstract}
\bigskip


\section{Introduction}
In this paper we investigate the existence of a strong solution to the initial value problem 
\begin{eqnarray}
	\pt u+(u\cdot\nabla) u+\nabla p&=&\Delta u-\Psi\nabla\phi\ \ \mbox{in $\rn\times(0,T)\equiv Q_T$},\label{nsf1}\\
	\pt c_i+\nabla\cdot(c_iu)&=&\Delta c_i+\nabla\cdot\left(z_ic_i\nabla\phi\right)\ \ \mbox{in $ Q_T$},\  i=1,\cdots, I,\label{nsf2}\\
	-\Delta\phi&=&\Psi\ \ \mbox{in $ Q_T$},\label{nsf3}\\
	\Psi&=&\sum_{i=1}^{I}z_ic_i,\label{nsfa}\\
	\nabla\cdot u&=&0\ \ \mbox{in $ Q_T$},\label{nsf4}\\
	u(x,0)&=& u^{(0)}(x),\ \ c_i(x,0)=c_i^{(0)}(x)\ \ \mbox{on $\rn$}\label{nsf5}.
\end{eqnarray}
This problem can be used to describe the transport and diffusion of ions in electrolyte solutions. In this case, $I$ is the number of  ionic species. For each $i\in\{1, \cdots, I\}$
$c_i$ is the ionic concentration of the $i$-th specie and $z_i$ is the corresponding valences. The vector field  $u \in \rn$
is fluid velocity, $p$ is
the pressure, and $\phi$ is the electric potential. The system \eqref{nsf1}-\eqref{nsf4} is often called the Nernst-Planck-Navier-Stokes (NPNS) system. It appears 
in the study of many physical and
biological processes \cite{ABL,BB}. Examples are  ion particles in the electrokinetic fluids \cite{ EHL,JJ} and
ion channels in cell membranes \cite{BTA,EI}, to name a couple. We refer the reader to \cite{R} for more information on the physical and biological relevance of the system.

Mathematical analysis of the NPNS system has attracted a lot of attentions recently. Most of the existing research deals with the case where the system is posed on a bounded domain with various types of boundary conditions. See \cite{CI,CIL,SC} and the references therein. However, problems with unbounded domains present different mathematical challenges from those with bounded ones. The initial value problem such as ours was first considered in \cite{J}, where local existence of a smooth solution was established via analytic semi-group theory. The so-called energy dissipation equalities associated with the system were obtained in \cite{LW}, from  which  a global-in-time weak solution was constructed. The objective of this paper is to improve on the regularity of the weak solution. Our main result is the following
\begin{theorem}\label{thm}Assume that
	\begin{eqnarray}
		|\uo|&\in &L^{\infty}(\rn)\cap  L^{2}(\rn)\ \ \mbox{and}\ \ \nabla\cdot\uo=0,\label{uoc}\\
		c_i^{(0)}&\in &L^{\infty}(\rn)\cap  L^{1}(\rn)\ \ \mbox{with} \ \ 	c_i^{(0)}\geq 0, \  i=1,\cdots, I.\label{cic}
	\end{eqnarray}
A result in  \cite{J} asserts that  there exists a local-in-time strong solution  $(c_1,\cdots, c_I, u, \phi)$ to \eqref{nsf1}-\eqref{nsf5}   with $c_i\geq 0$. Define
\begin{equation}\label{wdef}
	w=\sum_{i=1}^{I}c_i. 
\end{equation}
Then for each $\lz\geq 1$ there exist two positive numbers $C=C\left(N,z_1, \cdots, z_I, \lz\right)$ and $s_5=s_5(N, \lz)$ such that
\begin{eqnarray}
		\|w\|_{\infty, Q_T}&\leq& 16I\|w(\cdot,0)\|_{\infty, \rn}+	C\|w\|_{\lz, Q_T}^{s_5}.\label{cb}
%
		\end{eqnarray}	
\end{theorem}
Regularity properties of a global-in-time  weak solution as constructed in \cite{LW} are rather poor ( see \eqref{ee2} below).
As we shall see in the next section,   we do have
\begin{equation}\label{wes}
	\|w\|_{\frac{N+2}{N}, Q_T} \leq  c.
\end{equation} 
Here the constant $c$ is determined by $N, T, z_i,\|c_i^{(0)}\|_{1,\rn},\|c_i^{(0)}\|_{\infty,\rn}$, and $\|\uo\|_{2,\rn}$. Moreover, the dependence of  $c$ on $T$ is such that it becomes unbounded only when $T$ goes to infinity.
We can take $\lz=\frac{N+2}{N}$ in \eqref{cb} to derive that to each $T>0$ there corresponds a $c$ such that
\begin{equation}\label{cb1}
		\|w\|_{\infty, Q_T}\leq c.
\end{equation}
That is to say, $w$ never blows up in finite time. This, along with 
 \eqref{nsf3} and \eqref{il2} from  below, puts us in a position to apply the classical Calder\'{o}n-Zygmund estimate. Upon doing so,  we can obtain  that for each $s>1$ there is a constant $c$ with
\begin{equation}
	\sup_{0\leq t\leq T}	\|\phi(\cdot,t)\|_{W^{2,s}(\rn)}\leq c.\label{pls}
\end{equation}
Combing this with \eqref{cb1}, we further establish
\begin{equation}
	\|u\|_{\infty, Q_T}\leq c.\label{ub}
\end{equation}

 A strong solution is understood to be a weak one, as defined in \cite{LW}, with the additional properties \eqref{cb1}-\eqref{ub}.
 Obviously, under \eqref{cb1}-\eqref{ub} higher regularity of the solution can be obtained via a bootstrap argument. In fact, a strong solution can be shown to satisfy system \eqref{nsf1}-\eqref{nsf3} in the a.e. sense \cite{OP}. We will not pursue the details here. Once again, the key to our approach is that the constants $C$ and $s_5$ in \eqref{cb} do not depend $T$. This implies that a local-in-time strong solution never blows up in finite time. As a result, it can be extended as a global strong solution. Therefore, our estimate \eqref{cb} bridges the gap between a weak solution and a strong one.

Note that if $I=2$ then \eqref{cb} was already obtained in \cite{LW}. However, as noted in the article, the method employed there cannot be extended to the case where $I>2$. Also see \cite{ZL}. 

Our approach is based upon an idea  developed by the author in \cite{X, X1}. It combines suitable scaling of the dependent variables with a De Giorgi iteration scheme. It seems to be very effective in dealing with the type of non-linearity appearing in \eqref{nsf2}. 

This work is organized as follows. In Section \ref{sec2}, we collect some relevant known results, while Section \ref{sec3} is devoted to the proof of Theorem \ref{thm} and \eqref{ub}.

\section{Preliminary results}\label{sec2} In this section, we first make some preliminary analysis on \eqref{nsf1}-\eqref{nsf5}. Then we state a couple of relevant known results.

From here on we shall assume that our solution is a a local-in-time strong one. There are different ways one can establish the existence of such a solution. We already mention \cite{J}.  Our objective is to show that such a solution never blows up in finite time. 

The following lemma is a consequence of the energy dissipation equalities in \cite{LW}.
\begin{lemma}We have
	\begin{eqnarray}
		\lefteqn{\sup_{0\leq t\leq T}\left(\irn|u|^2dx+\irn|\nabla\phi|^2dx\right)+\irnt|\nabla u|^2dxdt+\irnt \Psi^2dxdt}\nonumber\\
		&&+\irnt  \sum_{i=1}^{I}z_i^2c_i|\nabla\phi|^2 dxdt\leq c\left(\|\uo\|_{2,\rn}^2+\left\|\sum_{i=1}^{I}z_ic_i^{(0)}\right\|_{\frac{2N}{N+2},\rn}^2\right).\label{ee2}
	\end{eqnarray}
		Here and in what follows the letter $c$, unless otherwise stated, denotes a generic positive constant that depends on $ I, N$, and $z_i$, i.e.,
	\begin{equation}\label{c0}
		c=c( I, N, z_i).
	\end{equation}
\end{lemma}
\begin{proof} We easily verify that
	\begin{equation*}
		(u\cdot\nabla) u\cdot u=\frac{1}{2}(u\cdot\nabla)|u|^2.
	\end{equation*}
This together with \eqref{nsf4} implies
\begin{equation*}
		\irn (u\cdot\nabla) u\cdot u\ dx=0.
\end{equation*}	
Similarly,
\begin{equation*}
	\irn (u\cdot\nabla) p\ dx=0.
\end{equation*}
With these in mind, we use $u$ as a test function in \eqref{nsf1} to deduce
\begin{equation}\label{nsf19}
	\frac{1}{2}\frac{d}{dt}\irn|u|^2dx+\irn|\nabla u|^2dx=-\irn \Psi\nabla\phi\cdot udx.
\end{equation}
To estimate the term on the right-hand side, we differentiate \eqref{nsf3} with respect to $t$ and use $\phi$ as a test function in the resulting equation to obtain
\begin{equation}\label{nsf15}
	\frac{1}{2}\frac{d}{dt}\irn|\nabla\phi|^2dx=\irn \pt\Psi\phi dx.
\end{equation}
Next, we use $z_i\phi$ as a test function in \eqref{nsf2} to get
\begin{eqnarray*}
	z_i\irn \phi\pt c_idx&=&z_i\irn c_iu\cdot  \nabla\phi dx-z_i\irn\nabla c_i\cdot  \nabla\phi dx-z_i^2\irn c_i|\nabla\phi|^2 dx\nonumber\\
	&=&z_i\irn c_iu\cdot  \nabla\phi dx-z_i\irn \Psi c_idx-z_i^2\irn c_i|\nabla\phi|^2 dx.
\end{eqnarray*}
The last step is due to \eqref{nsf3}.
 Sum up the equations  over $i$ to derive
\begin{equation*}
	\irn\pt\Psi\phi dx=\irn \Psi u\cdot \nabla\phi dx-\irn \Psi^2dx-\irn  \sum_{i=1}^{I}z_i^2c_i|\nabla\phi|^2 dx.
\end{equation*}
Substitute this into \eqref{nsf15} and add the resulting equation to \eqref{nsf19} to deduce
\begin{eqnarray*}
	\lefteqn{\frac{1}{2}\frac{d}{dt}\irn|u|^2dx+\irn|\nabla u|^2dx+\frac{1}{2}\frac{d}{dt}\irn|\nabla\phi|^2dx}\nonumber\\
	&&+\irn \Psi^2dx+\irn  \sum_{i=1}^{I}z_i^2c_i|\nabla\phi|^2 dx=0.
\end{eqnarray*}
After an integration, we arrive at
\begin{eqnarray}
	\lefteqn{\sup_{0\leq t\leq T}\left(\irn|u|^2dx+\irn|\nabla\phi|^2dx\right)+\irnt|\nabla u|^2dxdt+\irnt \Psi^2dxdt}\nonumber\\
	&&+\irnt  \sum_{i=1}^{I}z_i^2c_i|\nabla\phi|^2 dxdt\leq \frac{3}{2}\left(\irn|\uo(x)|^2dx+\irn|\nabla\phi(x,0)|^2dx\right).\label{ee1}
\end{eqnarray}
To bound the last term in the above inequality, we let $t=0$ in \eqref{nsf3} to get
\begin{equation*}
	-\Delta \phi(x,0)=\Psi(x,0)=\sum_{i=1}^{I}z_ic_i^{(0)}(x).
\end{equation*}
Use $\phi(x,0)$ as a test function to deduce
\begin{eqnarray}
	\irn|\nabla\phi(x,0)|^2dx&=&\irn \sum_{i=1}^{I}z_ic_i^{(0)}(x)\phi(x,0)dx\nonumber\\
	&\leq&\left(\irn\left|\sum_{i=1}^{I}z_ic_i^{(0)}(x)\right|^{\frac{2N}{N+2}}dx\right)^{\frac{N+2}{2N}}\left(\irn|\phi(x,0)|^{\frac{2N}{N-2}}dx\right)^{\frac{N-2}{2N}}.\label{dis}
\end{eqnarray}
Recall that the Sobolev inequality in the whole space asserts
\begin{equation}\label{sob}
	\|f\|_{\frac{2N}{N-2},\rn}\leq c(N)\|\nabla f\|_{2,\rn}\ \ \mbox{for each $f\in H^1(\rn)$}.
\end{equation}
This together with \eqref{dis} implies
\begin{equation*}
	\|\nabla\phi(\cdot,0)\|_{2,\rn}\leq c\left\|\sum_{i=1}^{I}z_ic_i^{(0)}\right\|_{\frac{2N}{N+2},\rn}.
\end{equation*}
Substitute this into \eqref{ee1} to complete the proof. Note that the right-hand side of \eqref{ee2} is finite due to our assumptions \eqref{uoc} and \eqref{cic}.
\end{proof}

\begin{proof}[Proof of \eqref{wes}] A similar result was obtained in \cite{LW} under some additional assumptions on the initial data, which will be removed here.
	
	Integrate \eqref{nsf2} over $\rn$ to get
	\begin{equation}\label{il1}
		\frac{d}{dt}\irn c_i\ dx=0\ \ \mbox{for $i=,\cdots, I$,}
	\end{equation}
from whence follows
\begin{equation}\label{il2}
	\sup_{0\leq t\leq T}\irn c_idx\leq \irn c_i^{(0)}dx.
\end{equation}

Before we continue, we must point out that $\ln c_i$  is not a legitimate test function for \eqref{nsf2} because it may not be bounded away from $0$ below. As a result, it is no longer a Sobolev function of the space variables. This point seems to have been overlooked in \cite{LW}. However,
for each $\ve>$ the function $\ln(c_i+\ve)-\ln\ve$ is. Upon using it, we
 derive
\begin{eqnarray}
	\lefteqn{	\frac{d}{dt}\irn\int_{0}^{c_i}[\ln(\mu+\ve)-\ln\ve]d\mu dx+\irn\frac{1}{c_i+\ve}|\nabla c_i|^2dx}\nonumber
	\\
	&=&\irn \frac{c_i}{c_i+\ve}(u\cdot\nabla) c_idx-z_i\irn \frac{c_i}{c_i+\ve}\nabla\phi\cdot\nabla c_idx.\label{ces1}
\end{eqnarray}
We can infer from \eqref{nsf4} that
\begin{equation*}
	\irn \frac{c_i}{c_i+\ve}(u\cdot\nabla) c_idx=\irn (u\cdot\nabla)\int_{0}^{c_i}\frac{s}{s+\ve}dsdx=0.
\end{equation*}
The last term in \eqref{ces1} can be estimated as follows:
\begin{eqnarray*}
-z_i\irn \frac{c_i}{c_i+\ve}\nabla\phi\cdot\nabla c_idx
	&\leq&\frac{1}{2}\irn\frac{1}{c_i+\ve}|\nabla c_i|^2dx+\frac{z_i^2}{2}\irn c_i|\nabla\phi|^2dx.
\end{eqnarray*}
Use the preceding two results in \eqref{ces1} and integrate the resulting inequality with respect to $t$ to derive
\begin{equation}\label{ces3}
\irnt\frac{1}{c_i+\ve}|\nabla c_i|^2dxdt\leq c\irn\int_{0}^{c_i^{(0)}}[\ln(\mu+\ve)-\ln\ve]d\mu dx	+cz_i^2\irnt c_i|\nabla\phi|^2dxdt.
\end{equation}
Here we have used the fact that
$$\int_{0}^{c_i}[\ln(\mu+\ve)-\ln\ve]d\mu \geq 0.$$
Note that
$$\ln(\mu+\ve)-\ln\ve=\ln\left(1+\frac{\mu}{\ve}\right)\leq \frac{\mu}{\ve}\ \ \mbox{for $\mu\geq0$ }.$$
This together with \eqref{ee2} and \eqref{cic} implies that 
\begin{equation}
	\mbox{the right hand side of \eqref{ces3}}\leq c(\ve).\nonumber
\end{equation} 
Obviously, $c(\ve) $ here depends on $\|c_i^{(0)}\|_{1,\rn},\|c_i^{(0)}\|_{\infty,\rn}$, and $\|\uo\|_{2,\rn}$ due to the upper bound in \eqref{ee2}.
Unfortunately, $c(\ve)$ blows up as $\ve\ra 0$. This will cause some complications. To circumvent them, 
we easily see that
$$\sqrt{c_i+\ve}-\sqrt{\ve}=\frac{c_i}{\sqrt{c_i+\ve}+\sqrt{\ve}}\leq \sqrt{c_i}$$
With this, 
 \eqref{sob}, \eqref{ces3}, and \eqref{il1} in mind, we calculate that
\begin{eqnarray}
\lefteqn{\int_{Q_T}\left(\sqrt{c_i+\ve}-\sqrt{\ve}\right)^{\frac{4}{N}+2}dxdt}\nonumber\\	&\leq&\int_{0}^{T}\left(\irn\left(\sqrt{c_i+\ve}-\sqrt{\ve}\right)^2dx\right)^{\frac{2}{N}}\left(\irn\left(\sqrt{c_i+\ve}-\sqrt{\ve}\right)^{\frac{2N}{N-2}}dx\right)^{\frac{N-2}{N}}dt\nonumber\\
	&\leq&\left(\sup_{0\leq t\leq T}\irn c_idx\right)^{\frac{2}{N}}\int_{0}^{T}\left(\irn\left(\sqrt{c_i+\ve}-\sqrt{\ve}\right)^{\frac{2N}{N-2}}dx\right)^{\frac{N-2}{N}}dt\nonumber\\
	&\leq&c\left(\sup_{0\leq t\leq T}\irn c_idx\right)^{\frac{2}{N}}\int_{0}^{T}\irn \frac{1}{c_i+\ve}|\nabla c_i|^2dxdt
	\leq c.\label{ns4}
\end{eqnarray}
It is elementary to show that
$$c_i=\left(\sqrt{c_i+1}-1\right)^2+2\left(\sqrt{c_i+1}-1\right)\leq 2\left(\sqrt{c_i+1}-1\right)^2+1,$$
from whence it follows that
\begin{eqnarray}
	\irnt c_i^{\frac{N+2}{N}}dxdt&=&\int_{\{c_i\leq 1\}}c_i^{\frac{N+2}{N}}dxdt+\int_{\{c_i> 1\}}c_i^{\frac{N+2}{N}}dxdt\nonumber\\
	&\leq&\irnt c_idxdt+c\irnt\left(\sqrt{c_i+1}-1\right)^{\frac{2(N+2)}{N}} dxdt+c|\{c_i> 1\}|\nonumber\\
	&\leq&cT+c.\label{ces2}
\end{eqnarray}
The last step is due to \eqref{il2} and \eqref{ns4}.
\end{proof}
We would like to remark that estimate \eqref{ces2} is the only place where the constant $c$ depends on $T$. But this does not affect our global existence because the constant blows up only when $T\ra \infty$. 

The following lemma is the foundation of a De Giorgi iteration scheme, whose proof can be found in (\cite{D}, p.12).
\begin{lemma}\label{ynb}
	Let $\{y_n\}, n=0,1,2,\cdots$, be a sequence of positive numbers satisfying the recursive inequalities
	\begin{equation*}
		y_{n+1}\leq cb^ny_n^{1+\alpha}\ \ \mbox{for some $b>1, c, \alpha\in (0,\infty)$.}
	\end{equation*}
	If
	\begin{equation*}
		y_0\leq c^{-\frac{1}{\alpha}}b^{-\frac{1}{\alpha^2}},
	\end{equation*}
	then $\lim_{n\rightarrow\infty}y_n=0$.
\end{lemma}

 The following lemma is essentially a consequence of the interpolation inequality for $L^q$ norms (\cite{GT}, p.146). It says that when we decrease  q in the $L^q$-norm its exponent increases in the situation under our consideration.
\begin{lemma}\label{prop}
	Let $f\in L^\ell(Q_T)\cap L^\infty(Q_T)$ for some $\ell\geq 1$. Assume that there exist $q\in (\ell, \infty)$, $\delta>0$, and $c>0$ such  that
	\begin{equation*}
		\|f\|_{\infty,Q_T}\leq c	\|f\|_{q,Q_T}^{1+\delta}.
	\end{equation*}
	If
	\begin{equation}\label{pro2}
		\delta<\frac{\ell}{q-\ell},
	\end{equation}	then
	\begin{equation}\label{pro1}
		\|f\|_{\infty,Q_T}\leq c^{\frac{q}{\ell(1+\delta)-q\delta}}\|f\|_{\ell,Q_T}^{\frac{\ell(1+\delta)}{\ell(1+\delta)-q\delta}}.
	\end{equation}
\end{lemma}
\begin{proof}We easily check
	\begin{eqnarray}
		\|f\|_{\infty,Q_T}&\leq& c\left(\irnt |f|^{q-\ell+\ell}dxdt\right)^{\frac{1+\delta}{q}}\nonumber\\
		&\leq&		c\|f\|_{\infty,Q_T}^{\left(1-\frac{\ell}{q}\right)\left(1+\delta\right)}\|f\|_{\ell,Q_T}^{\frac{\ell}{q}\left(1+\delta\right)}.\label{pro4}
	\end{eqnarray}
	Condition \eqref{pro2} implies
	$$\left(1-\frac{\ell}{q}\right)\left(1+\delta\right)<1.$$
	As a result, we can factor out $\|f\|_{\infty,Q_T}^{\left(1-\frac{\ell}{q}\right)\left(1+\delta\right)}$ from  \eqref{pro4}, thereby obtaining  \eqref{pro1}.
\end{proof}
We obviously have
\begin{equation}
	\frac{\ell(1+\delta)}{\ell(1+\delta)-q\delta}>1+\delta.\nonumber
\end{equation}

Finally, the following two inequalities will be used without acknowledgment:
\begin{eqnarray*}
	(|a|+|b|)^\gamma&\leq&\left\{\begin{array}{ll}
		2^{\gamma-1}(|a|^\gamma+|b|^\gamma)&\mbox{if $\gamma\geq 1$},\\
		|a|^\gamma+|b|^\gamma&\mbox{if $\gamma\leq 1$}.
	\end{array}\right.
\end{eqnarray*}
.

\section{Proof of Theorem \ref{thm}}\label{sec3}


The proof of \eqref{cb} is the core of our development. It is inspired by an idea from \cite{X1}. We scale the relevant equations by an appropriate $L^q$ norm and then apply an iteration scheme of the De Giorgi type. Various parameters are introduced in the process. Desired estimates can be attained via careful selections of these parameters. That constitutes the novelty of our approach.


\begin{proof}[Proof of \eqref{cb} ] 
	Let $w$ be given as in \eqref{wdef}. That is,
	$$w=\sum_{i=1}^{I}c_i.$$
	Set
	\begin{equation}\label{pdef}
		\vpi=\frac{c_i}{\|w\|_{r,Q_T}},
	\end{equation}
	where $r\in [1,\infty)$, whose precise value remains to be determined.
Divide through  \eqref{nsf2}  by $\|w\|_{r,Q_T}$ to obtain
	\begin{eqnarray}
		\pt \vpi+\nabla\cdot(\vpi u)&=&\Delta \vpi+\nabla\cdot\left(z_i\vpi\nabla\phi\right)\ \ \mbox{in $ Q_T$},\  i=1,\cdots, I.\label{ns2}
	\end{eqnarray}
Choose
	\begin{equation}\label{kcon1}
		k\geq2\max_{1\leq i\leq I}\|\vpi(\cdot,0)\|_{\infty,\rn}.
		\end{equation}
	as below.	Define
	\begin{eqnarray}
		k_n&=&k-\frac{k}{2^{n+1}} \ \ \mbox{for $n=0,1,\cdots$.} \label{4kcon3}
	\end{eqnarray}
			Fix
				\begin{equation*}
					\beta>1.
				\end{equation*}
	 Then it is easy to check  that the function
$$\left(\frac{1}{ k_{n}^\beta}-\frac{1}{\vpi^\beta}\right)^+$$
is a legitimate test function  for \eqref{ns2}.  
Upon using it, we obtain
\begin{eqnarray}
	\lefteqn{	\frac{d}{dt}\irn\int_{k_{n}}^{\vpi}\left(\frac{1}{ k_{n}^\beta}-\frac{1}{\mu^\beta}\right)^+d\mu   dx+\beta\irnn\frac{1}{\vpi^{1+\beta}}\left|\nabla \vpi\right|^2 dx}\nonumber\\
	&=&\beta\irnn\frac{1}{\vpi^{\beta}}  (u\cdot\nabla)  \vpi dx-\beta z_i \irnn\frac{1}{\vpi^{\beta}} \nabla\phi\cdot\nabla \vpi dx,
\label{ub1}
\end{eqnarray}
where
\begin{equation}
	\Omega_n(t)=\{x\in\rn: \vpi(x,t)\geq k_n\}.\nonumber
\end{equation}
Evidently,
\begin{eqnarray*}
	\beta\irnn\frac{1}{\vpi ^{1+\beta}}\left|\nabla \vpi \right|^2dx&=&\frac{4\beta}{(\beta-1)^2}\irn\left|\nabla\left(\frac{1}{ k_{n}^{\frac{\beta-1}{2}}}-\frac{1}{\vpi ^{\frac{\beta-1}{2}}}\right)^+\right|^2dx.
\end{eqnarray*}
	Note from \eqref{nsf4} that
	\begin{eqnarray}
\beta	\irnn\frac{1}{\vpi ^{\beta}}  (u\cdot\nabla)  \vpi  dx	&=&\frac{\beta}{\beta-1}	\irn (u\cdot\nabla) \left(\frac{1}{k_{n}^{\beta-1}}-\frac{1}{\vpi ^{\beta-1}}\right)^+ dx=0,\label{4ub2}
	\end{eqnarray}
while \eqref{nsf3} asserts that
\begin{eqnarray}
-\beta \irnn\frac{1}{\vpi ^{\beta}} \nabla\phi\cdot\nabla \vpi dx&=&-\frac{\beta}{\beta-1}	\irn \nabla\phi\cdot\nabla \left(\frac{1}{k_{n}^{\beta-1}}-\frac{1}{\vpi ^{\beta-1}}\right)^+ dx\nonumber\\
&=&-\frac{\beta}{\beta-1}	\irn \Psi \left(\frac{1}{k_{n}^{\beta-1}}-\frac{1}{\vpi ^{\beta-1}}\right)^+ dx\nonumber\\
&=&-\frac{\beta}{\beta-1}	\irn \Psi  \left(\frac{1}{ k_{n}^{\frac{\beta-1}{2}}}-\frac{1}{\vpi ^{\frac{\beta-1}{2}}}\right)^+\left(\frac{1}{ k_{n}^{\frac{\beta-1}{2}}}+\frac{1}{\vpi ^{\frac{\beta-1}{2}}}\right)dx\nonumber\\
&\leq&\frac{2\beta}{(\beta-1) k_{n}^{\frac{\beta-1}{2}}}\irn |\Psi|  \left(\frac{1}{ k_{n}^{\frac{\beta-1}{2}}}-\frac{1}{\vpi ^{\frac{\beta-1}{2}}}\right)^+dx
.\label{4ub3}
\end{eqnarray}
	We next claim 
	\begin{equation}\label{hope10}
		\int_{k_{n}}^{\vpi }\left(\frac{1}{k_{n}^\beta}-\frac{1}{\mu^\beta}\right)^+d\mu\geq \frac{2\beta}{(1-\beta)^2} \left[\left(\frac{1}{k_{n}^{\frac{\beta-1}{2}}}-\frac{1}{\vpi ^{\frac{\beta-1}{2}}}\right)^+\right]^2.
	\end{equation}
	To see this,  we compute
	\begin{eqnarray}
		\left(\int_{k_{n}}^{\vpi }\left(\frac{1}{k_{n}^\beta}-\frac{1}{\mu^\beta}\right)^+d\mu\right)^{\prime}&=&\left(\frac{1}{k_{n}^\beta}-\frac{1}{\varphi_i^\beta}\right)^+=\frac{1}{k_{n}^\beta}-\frac{1}{\varphi_i^\beta}\ \ \mbox{for $\vpi \geq k_{n}$},\nonumber\\
		\left(\int_{k_{n}}^{\vpi }\left(\frac{1}{k_{n}^\beta}-\frac{1}{\mu^\beta}\right)^+d\mu\right)^{\prime\prime}&=&\beta \vpi ^{-1-\beta}\ \ \mbox{for $\vpi \geq k_{n}$},\nonumber\\
	\left.	\int_{k_{n}}^{\vpi }\left(\frac{1}{k_{n}^\beta}-\frac{1}{\mu^\beta}\right)^+d\mu\right|_{\vpi=k_n}&=&\left.\left(\int_{k_{n}}^{\vpi }\left(\frac{1}{k_{n}^\beta}-\frac{1}{\mu^\beta}\right)^+d\mu\right)^{\prime}\right|_{\vpi=k_n}=0.\label{ic1}
			\end{eqnarray}
			Similarly, 
				\begin{eqnarray}
					\left(\frac{2\beta}{(1-\beta)^2}\left[\left(\frac{1}{k_{n}^{\frac{\beta-1}{2}}}-\frac{1}{\vpi ^{\frac{\beta-1}{2}}}\right)^+\right]^2\right)^{\prime}&=&\frac{2\beta}{\beta-1}\left(\frac{1}{k_{n}^{\frac{\beta-1}{2}}}-\frac{1}{\vpi ^{\frac{\beta-1}{2}}}\right)\vpi ^{-\frac{1+\beta}{2}}\ \ \mbox{for $\vpi \geq k_{n}$},\nonumber\\
		\left(\frac{2\beta}{(1-\beta)^2}\left[\left(\frac{1}{k_{n}^{\frac{\beta-1}{2}}}-\frac{1}{\vpi ^{\frac{\beta-1}{2}}}\right)^+\right]^2\right)^{\prime\prime}&=&\frac{2\beta}{\beta-1}\left(\beta-\frac{1+\beta}{2}\left(\frac{\vpi }{k_{n}}\right)^{\frac{\beta-1}{2}}\right)\vpi ^{-1-\beta}\ \ \mbox{for $\vpi \geq k_{n}$},\nonumber\\
	\left.\left[\left(\frac{1}{k_{n}^{\frac{\beta-1}{2}}}-\frac{1}{\vpi ^{\frac{\beta-1}{2}}}\right)^+\right]^2\right|_{\vpi=k_n}&=&\left.	\left(\left[\left(\frac{1}{k_{n}^{\frac{\beta-1}{2}}}-\frac{1}{\vpi ^{\frac{\beta-1}{2}}}\right)^+\right]^2\right)^{\prime}\right|_{\vpi=k_n}=0.\label{ic2}
	\end{eqnarray}
	We can easily verify  that
	\begin{equation*}
		\left(\int_{k_{n}}^{\vpi }\left(\frac{1}{k_{n}^\beta}-\frac{1}{\mu^\beta}\right)^+d\mu\right)^{\prime\prime}\geq \left(\frac{2\beta}{(1-\beta)^2}\left[\left(\frac{1}{k_{n}^{\frac{\beta-1}{2}}}-\frac{1}{\vpi ^{\frac{\beta-1}{2}}}\right)^+\right]^2\right)^{\prime\prime}\ \ \mbox{for $\vpi \geq k_{n}$.}
	\end{equation*}
	Integrate this inequality  twice and keep in mind \eqref{ic1} and \eqref{ic2}
	to obtain \eqref{hope10}.
	
	Recall \eqref{4kcon3} and \eqref{kcon1} to derive
	\begin{equation}\label{ub3}
		\left.\int_{k_{n}}^{\vpi }\left(\frac{1}{k_{n}^\beta}-\frac{1}{\mu^\beta}\right)^+d\mu\right|_{t=0}=0.
	\end{equation}
Use \eqref{4ub3} and  \eqref{4ub2} in \eqref{ub1},
integrate the resulting inequality with respect to $t$, and keep in mind \eqref{ub3} and \eqref{hope10} to deduce
	\begin{eqnarray}
	\lefteqn{\sup_{0\leq t\leq T}\irn\left[\left(\frac{1}{k_{n}^{\frac{\beta-1}{2}}}-\frac{1}{\vpi ^{\frac{\beta-1}{2}}}\right)^+\right]^2dx}\nonumber\\
&&	+\irnt\left|\nabla\left(\frac{1}{ k_{n}^{\frac{\beta-1}{2}}}-\frac{1}{\vpi ^{\frac{\beta-1}{2}}}\right)^+\right|^2dxdt		\leq \frac{c}{k^{\frac{\beta-1}{2}}}\iqt |\Psi|  \left(\frac{1}{ k_{n}^{\frac{\beta-1}{2}}}-\frac{1}{\vpi ^{\frac{\beta-1}{2}}}\right)^+dxdt.\nonumber
	\end{eqnarray}
Set
	\begin{equation}\label{qn}
		Q_n=\{(x,t)\in Q_T:\vpi(x,t)\geq k_n\}.
	\end{equation}
Define
\begin{equation}
	y_n=|Q_n|.\nonumber
\end{equation}
%
We proceed to show that $\{y_n\}$ satisfies the condition in Lemma \ref{ynb}. 
By calculations similar to those in \eqref{ns4},  we have
\begin{eqnarray}
	\lefteqn{	\irnt\left[\left(\frac{1}{ k_{n}^{\frac{\beta-1}{2}}}-\frac{1}{\vpi ^{\frac{\beta-1}{2}}}\right)^+\right]^{\frac{4}{N}+2}dxdt}\nonumber\\
	&\leq&\int_{0}^{T}\left(\irn\left[\left(\frac{1}{ k_{n}^{\frac{\beta-1}{2}}}-\frac{1}{\vpi ^{\frac{\beta-1}{2}}}\right)^+\right]^{2}dx \right)^{\frac{2}{N}}\left(\irn\left[\left(\frac{1}{ k_{n}^{\frac{\beta-1}{2}}}-\frac{1}{\vpi ^{\frac{\beta-1}{2}}}\right)^+\right]^{\frac{2N}{N-2}}dx\right)^{\frac{N-2}{N}}dt\nonumber\\
	&\leq& c\left(\sup_{0\leq t\leq T}\irn\left[\left(\frac{1}{ k_{n}^{\frac{\beta-1}{2}}}-\frac{1}{\vpi ^{\frac{\beta-1}{2}}}\right)^+\right]^{2}dx \right)^{\frac{2}{N}}\irnt\left|\nabla\left(\frac{1}{ k_{n}^{\frac{\beta-1}{2}}}-\frac{1}{\vpi ^{\frac{\beta-1}{2}}}\right)^+\right|^2dxdt\nonumber\\
	&\leq& c\left(\frac{c}{k^{\frac{\beta-1}{2}}}\iqt |\Psi|  \left(\frac{1}{ k_{n}^{\frac{\beta-1}{2}}}-\frac{1}{\vpi ^{\frac{\beta-1}{2}}}\right)^+dxdt\right)^{\frac{N+2}{N}}\nonumber\\
&\leq& c\left(\frac{c}{k^{\frac{\beta-1}{2}}}\left(\int_{Q_{n}}|\Psi|^{\frac{2(N+2)}{N+4}}dxdt\right)^{\frac{N+4}{2(N+2)}}\left[\iqt\left[ \left(\frac{1}{ k_{n}^{\frac{\beta-1}{2}}}-\frac{1}{\vpi ^{\frac{\beta-1}{2}}}\right)^+\right]^{\frac{2(N+2)}{N}}dxdt\right]^{\frac{N}{2(N+2)}}\right)^{\frac{N+2}{N}}	,\nonumber
\end{eqnarray}
from whence it follows
\begin{eqnarray}
\lefteqn{\left[\iqt\left[ \left(\frac{1}{ k_{n}^{\frac{\beta-1}{2}}}-\frac{1}{\vpi ^{\frac{\beta-1}{2}}}\right)^+\right]^{\frac{2(N+2)}{N}}dxdt\right]^{\frac{1}{2}}}\nonumber\\
&\leq&\frac{c}{k^{\frac{(\beta-1)(N+2)}{2N}}}\left(\int_{Q_{n}}|\Psi|^{\frac{2(N+2)}{N+4}}dxdt\right)^{\frac{N+4}{2N}}.\label{rub3}
\end{eqnarray}
It is easy to verify that
\begin{eqnarray*}
	\irnt\left[\left(\frac{1}{ k_{n}^{\frac{\beta-1}{2}}}-\frac{1}{\vpi ^{\frac{\beta-1}{2}}}\right)^+\right]^{\frac{4}{N}+2}dxdt&\geq&
	\int_{Q_{n+1}}\left[\left(\frac{1}{ k_{n}^{\frac{\beta-1}{2}}}-\frac{1}{\vpi ^{\frac{\beta-1}{2}}}\right)^+\right]^{\frac{4}{N}+2}dxdt\nonumber\\
	&\geq&\left(\frac{1}{ k_{n}^{\frac{\beta-1}{2}}}-\frac{1}{k_{n+1}^{\frac{\beta-1}{2}}}\right)^{\frac{4}{N}+2}|Q_{n+1}|
	\nonumber\\
		&=&\left(\frac{\left(1-\frac{1}{2^{n+2}}\right)^{\frac{\beta-1}{2}}-\left(1-\frac{1}{2^{n+1}}\right)^{\frac{\beta-1}{2}}}{ k^{\frac{\beta-1}{2}}\left(1-\frac{1}{2^{n+1}}\right)^{\frac{\beta-1}{2}}\left(1-\frac{1}{2^{n+2}}\right)^{\frac{\beta-1}{2}}}\right)^{\frac{4}{N}+2}|Q_{n+1}|
	\nonumber\\
	&	\geq&\frac{c|Q_{n+1}|}{2^{(\frac{4}{N}+2)n}k^{\frac{(\beta-1)(N+2)}{N}}}.
\end{eqnarray*}
Combining this with \eqref{rub3} yields
\begin{equation}
	y_{n+1}\leq c4^{n}\left(\int_{Q_{n}}|\Psi|^{\frac{2(N+2)}{N+4}}dxdt\right)^{\frac{N+4}{N}}.\label{rub4}
\end{equation}
	Fix
\begin{eqnarray}
	q>1+\frac{N}{2}.\label{qcon}
\end{eqnarray}
Then we have
\begin{eqnarray*}
	\left(\int_{Q_{n}}|\Psi|^{\frac{2(N+2)}{N+4}}dxdt\right)^{\frac{N+4}{N}}&=&\|\Psi\|_{\frac{2(N+2)}{N+4},Q_n}^{\frac{2(N+2)}{N}}\nonumber\\
	&\leq&|Q_n|^{\frac{2(N+2)}{N}\left[\frac{N+4}{2(N+2)}-\frac{1}{q}\right]}\|\Psi\|_{q,Q_T}^{\frac{2(N+2)}{N}}\nonumber\\
	&\leq&c|Q_n|^{\frac{2(N+2)}{N}\left[\frac{N+4}{2(N+2)}-\frac{1}{q}\right]}\|w\|_{q,Q_T}^{\frac{2(N+2)}{N}}.
\end{eqnarray*}
The last step here is due to \eqref{nsfa}.
Use this in \eqref{rub4} to derive
\begin{eqnarray}
	y_{n+1}
	&\leq &c 4^{n}\|w\|_{q,Q_T}^{\frac{2(N+2)}{N}}y_n^{1+\alpha},\label{ns11}
\end{eqnarray}
where 
\begin{eqnarray}
	\alpha&=&\frac{2}{N}\left(2-\frac{N+2}{q}\right) >0\ \ \mbox{due to  \eqref{qcon}}.\label{adef}
\end{eqnarray}
Now we pick a number
\begin{equation}
\ell>r.\nonumber
\end{equation}
Choose $k$ so large that
\begin{equation}\label{kcon6}
\max\left\{L_1\left\|\sum_{i=1}^{I}\vpi\right\|_{\ell,Q_T}^{\frac{\ell}{\ell-r}}, L_2\|w\|_{r,Q_T}^{-1}\|w\|_{ q,Q_T}^{\frac{q}{q-\lz}}\right\}	\leq k,
\end{equation}
where $L_1$ and $L_2$ are two positive numbers to be determined.  Note that the exponent of $\|\vpi\|_{\ell,Q_T}$ in the above inequality is the number $1+\delta$ in Lemma \ref{prop} as we can easily see from \eqref{pro4} that
\begin{eqnarray}
\left\|\sum_{i=1}^{I}\vpi\right\|	_{\ell,Q_T}^{\frac{\ell}{\ell-r}}&=&\left\|\frac{w}{\wnr}\right\|	_{\ell,Q_T}^{\frac{\ell}{\ell-r}}\nonumber\\
&\leq&\frac{1}{\wnr^{\frac{\ell}{\ell-r}}}\left[\left\|w\right\|	_{\infty,Q_T}^{\frac{\ell-r}{\ell}}\left\|w\right\|	_{r,Q_T}^{\frac{r}{\ell}}\right]^{\frac{\ell}{\ell-r}}
	=\left\|\sum_{i=1}^{I}\vpi\right\|	_{\infty,Q_T} .\label{jm1}
\end{eqnarray}
Here we have applied \eqref{pdef}. The selection of the exponent of $\|w\|_{ q,Q_T}$ in \eqref{kcon6} is based upon the same idea.
For each 
\begin{equation*}
	j>0
\end{equation*}
we have from \eqref{kcon6} that 
\begin{equation*}
	L_1^{j\alpha}\left\|\sum_{i=1}^{I}\vpi\right\|_{\ell,Q_T}^{\frac{j\alpha\ell}{\ell-r}}\leq k^{j\alpha},\ \ L_2^{j\alpha}\|w\|_{r,Q_T}^{-j\alpha}\|w\|_{ q,Q_T}^{\frac{j\alpha q}{q-\lz}}\leq k^{j\alpha}	.
\end{equation*}
Use this in \eqref{ns11} to deduce
\begin{equation*}
	y_{n+1}\leq \frac{c 4^n\|w\|_{r,Q_T}^{j\alpha}\|w\|_{ q,Q_T}^{b} k^{2j\alpha}}{(L_1L_2)^{j\alpha}\left\|\sum_{i=1}^{I}\vpi\right\|_{\ell,Q_T}^{\frac{j\alpha\ell}{\ell-r}}} y_n^{1+\alpha},
\end{equation*}
where
\begin{equation}\label{bdef}
	b=\frac{2(N+2)}{N}-\frac{j\alpha q}{q-\lz}.
\end{equation}
The introduction of $j$ here is very crucial. As we shall see, by choosing $j$ suitably large, we can make certain exponents in our nonlinear terms negative. This will enable us to balance out large positive exponents. 

To apply Lemma \ref{ynb}, we first recall \eqref{4kcon3}, \eqref{qn}, and \eqref{pdef} to deduce
\begin{eqnarray*}
	y_0&=&
	|Q_0|\leq \irnt\left(\frac{2\vpi}{k}\right)^{r}dxdt\leq\frac{2^r}{k^r}.
\end{eqnarray*}
Assume that
\begin{equation}\label{rcon}
	r>2j.
\end{equation}
Subsequently, we can pick $k$ so large that
\begin{eqnarray}
	\frac{2^r}{k^{r-2j}}\leq \frac{(L_1L_2)^{j}\left\|\sum_{i=1}^{I}\vpi\right\|_{\ell,Q_T}^{\frac{j\ell}{\ell-r}}}{ c^{\frac{1}{\alpha}}4^{\frac{1}{\alpha^2}}\|w\|_{r,Q_T}^{j}\|w\|_{ q,Q_T}^{\frac{b}{\alpha}}}.\label{jm7}
\end{eqnarray}
Lemma \ref{ynb} asserts 
$$\lim_{n\ra \infty}y_n=|\{\vpi\geq k\}|=0.$$ 
That is,
\begin{equation}\label{jm8}
	\sup_{Q_T}\sum_{i=1}^{I}	\vpi\leq I k.
\end{equation}
According to \eqref{kcon1},  \eqref{kcon6}, and \eqref{jm7}, it is enough for us to take
\begin{eqnarray}
	k&=&2\max_{1\leq i\leq I}\|\vpi(\cdot,0)\|_{\infty,\rn}+L_1 \left\|\sum_{i=1}^{I}\vpi\right\|_{\ell,Q_T}^{\frac{\ell}{\ell-r}}+L_2\|w\|_{r,Q_T}^{-1}\|w\|_{ q,Q_T}^{\frac{q}{q-\lz}}\nonumber\\ &&+ 2^{\frac{r}{r-2j}}c^{\frac{1}{\alpha(r-2j)}}4^{\frac{1}{\alpha^2(r-2j)}}(L_1L_2)^{-\frac{j}{r-2j}}\left\|\sum_{i=1}^{I}\vpi\right\|_{\ell,Q_T}^{-\frac{j\ell}{(r-2j)(\ell-r)}}\|w\|_{r,Q_T}^{\frac{j}{r-2j}}\|w\|_{ q,Q_T}^{\frac{b}{\alpha(r-2j)}}.\nonumber
\end{eqnarray}
Plug this into \eqref{jm8},  take $ L_1=\frac{1}{2I}$ in the resulting inequality,  and make use of \eqref{jm1} 
 to yield
\begin{eqnarray*}
\left\|\sum_{i=1}^{I}\vpi\right\|_{\infty,Q_T}&\leq& 4I\max_{1\leq i\leq I}\left\{\|\vpi(\cdot,0)\|_{\infty,\rn}\right\}+2IL_2\|w\|_{r,Q_T}^{-1}\|w\|_{ q,Q_T}^{\frac{q}{q-\lz}}\nonumber\\ &&+ C_1L_2^{-\frac{j}{r-2j}}\left\|\sum_{i=1}^{I}\vpi\right\|_{\ell,Q_T}^{-\frac{j\ell}{(r-2j)(\ell-r)}}\|w\|_{r,Q_T}^{\frac{j}{r-2j}}\|w\|_{ q,Q_T}^{\frac{b}{\alpha(r-2j)}},
\end{eqnarray*}
where 
\begin{equation}
	C_1=2^{\frac{r}{r-2j}}c^{\frac{1}{\alpha(r-2j)}}4^{\frac{1}{\alpha^2(r-2j)}}(2I)^{\frac{j}{r-2j}+1}.\nonumber
\end{equation}
The constant $c$ here  is understood as in  \eqref{c0}. Hence, 
it is independent of $j, r$, and $\ell$.
Recall \eqref{pdef} to deduce
\begin{eqnarray}
	\|w\|_{\infty,Q_T}&\leq& 4I\|w(\cdot,0)\|_{\infty,\rn}+2IL_2\|w\|_{ q,Q_T}^{\frac{q}{q-\lz}}\nonumber\\ &&+ C_1L_2^{-\frac{j}{r-2j}}\|w\|_{\ell,Q_T}^{-\frac{j\ell}{(r-2j)(\ell-r)}}\|w\|_{r,Q_T}^{1+\frac{j}{r-2j}+\frac{j\ell}{(r-2j)(\ell-r)}}\|w\|_{ q,Q_T}^{\frac{b}{\alpha(r-2j)} }.\label{est1}
\end{eqnarray}
We further require
\begin{equation}
	q>\lz.\nonumber
\end{equation}
Consequently, we can form the interpolation inequality 
\begin{equation}\label{intq}
	\|w\|_{q,Q_T}\leq \|w\|_{\infty,Q_T}^{\frac{q-\lz}{q}}\|w\|_{\lz,Q_T}^{\frac{\lz}{q}}.
\end{equation}
Choose $L_2$  so that
\begin{equation*}
	2IL_2\|w\|_{\lz,Q_T}^{\frac{\lz}{q-\lz}}= \frac{1}{2}.
\end{equation*}
As a result, the third term in \eqref{est1} can be estimated as follows:
\begin{equation}
	2IL_2\|w\|_{ q,Q_T}^{\frac{q}{q-\lz}}\leq 2IL_2\|w\|_{\infty,Q_T}\|w\|_{\lz,Q_T}^{\frac{\lz}{q-\lz}}= \frac{1}{2}\|w\|_{\infty,Q_T}.\nonumber
\end{equation}
With these in mind, we deduce from \eqref{est1} that
\begin{eqnarray}
	\|w\|_{\infty,Q_T}&\leq& 8I\|w(\cdot,0)\|_{\infty,\rn}\nonumber\\
	&&+ C\|w\|_{\lz,Q_T}^{s_1}\|w\|_{\ell,Q_T}^{-\frac{j\ell}{(r-2j)(\ell-r)}}\|w\|_{r,Q_T}^{1+\frac{j}{r-2j}+\frac{j\ell}{(r-2j)(\ell-r)}}\|w\|_{q,Q_T}^{\frac{b}{\alpha(r-2j)}}.\label{est2}
\end{eqnarray}
where
\begin{eqnarray}
	C&=&C_1(4I)^{\frac{j}{r-2j}}=2^{\frac{r}{r-2j}}c^{\frac{1}{\alpha(r-2j)}}4^{\frac{1}{\alpha^2(r-2j)}}(2I)^{\frac{j}{r-2j}+1}(4I)^{\frac{j}{r-2j}},\label{cc}\\
	s_1&=&\frac{\lz j}{(q-\lz)(r-2j)}.\label{s11}
\end{eqnarray}
We proceed to show that we can extract enough information from this inequality by making suitable choice of the parameters. The idea is to transform the last three different norms in  \eqref{est2} into a single one.

\begin{clm}Let $q$ be given as before, i.e.,
	\begin{equation}
		q>\max\left\{\frac{N+2}{2}, \lz\right\}.\nonumber
	\end{equation} 
	Define
	\begin{equation}
		\mq=\frac{2(N+2)(q-\lz)}{N\alpha q}.\label{mqd}
	\end{equation}
	 Then for each $r>\mq$ there holds
	\begin{equation}\label{wif}
		\|w\|_{\infty,Q_T}\leq 8I\|w(\cdot,0)\|_{\infty,\rn}+C_2\|w\|_{\lz,Q_T}^{s_2}\|w\|_{r,Q_T}^{\frac{r}{r-M_q}},
	\end{equation}
	where 
	\begin{eqnarray}
		s_2&=&s_1|_{j=\frac{\mq}{2}}\frac{\lz \mq}{(q-\lz)(r-\mq)},\nonumber\\
		C_2&=&C|_{j=\frac{\mq}{2}}
		=2^{\frac{r}{r-\mq}}c^{\frac{1}{\alpha(r-\mq)}}4^{\frac{1}{\alpha^2(r-\mq)}}(2I)^{\frac{\mq}{2(r-\mq)}+1}(4I)^{\frac{\mq}{2(r-\mq)}}.\nonumber
	\end{eqnarray}
\end{clm}
\begin{proof} According to \eqref{cc} and \eqref{s11},  $C$ and $s_1$ in \eqref{est2} do not depend on $\ell$. With this in mind, we take $\ell\ra \infty$ there  to derive
	\begin{eqnarray}
		\|w\|_{\infty,Q_T}&\leq& 8I\|w(\cdot,0)\|_{\infty,\rn}+ C\|w\|_{\lz,Q_T}^{s_1}\|w\|_{\infty,Q_T}^{-\frac{j}{(r-2j)}}\|w\|_{r,Q_T}^{\frac{r}{r-2j}}\|w\|_{q,Q_T}^{\frac{b}{\alpha(r-2j)}}.\label{e2}
	\end{eqnarray}
We still have \eqref{intq}.	Raise both sides of the inequality to the power of $\frac{jq}{(r-2j)(q-\lz)}$ to get
	\begin{equation}
		\|w\|_{q,Q_T}^{\frac{jq}{(r-2j)(q-\lz)}}\leq \|w\|_{\infty,Q_T}^{\frac{j}{(r-2j)}}\|w\|_{\lz,Q_T}^{\frac{j\lz}{(r-2j)(q-\lz)}}.\nonumber
	\end{equation}
	Incorporating this into \eqref{e2} yields
	\begin{eqnarray}
		\|w\|_{\infty,Q_T}&\leq& 8I\|w(\cdot,0)\|_{\infty,\rn}+C\|w\|_{\lz,Q_T}^{s_1}\|w\|_{r,Q_T}^{\frac{r}{r-2j}}\|w\|_{q,Q_T}^{\frac{b}{\alpha(r-2j)}-\frac{jq}{(r-2j)(q-\lz)}}.\label{e3}
	\end{eqnarray}
	We choose our parameters in such a way that the last exponent is $0$. That is,
	\begin{equation}
		\frac{b}{\alpha(r-2j)}-\frac{jq}{(r-2j)(q-\lz)}=0.\nonumber
	\end{equation}
	Plug \eqref{bdef} into this to get
		\begin{equation}
\frac{(N+2)}{N\alpha}	-\frac{jq}{(q-\lz)}=0,\nonumber
	\end{equation}
from whence it follows
\begin{equation}
	j=\frac{(N+2)(q-\lz)}{N\alpha q}=\frac{\mq}{2}.\nonumber
\end{equation}	
Substitute this into \eqref{e3} to  arrive at \eqref{wif}. The proof is complete.
	\end{proof}
 An easy consequence of the preceding claim is that
	\begin{equation}
w\in L^{\infty}(Q_T)\ \ \mbox{whenever}	\ \ \lz>\frac{N+2}{2}\ \ \mbox{and}\ \ w\in L^{\lz}(Q_T) .\nonumber
	\end{equation}
	Indeed, we easily see from \eqref{adef} that
	\begin{equation}\label{mql}
		\lim_{q\ra\infty}\mq=\frac{N+2}{2}.
	\end{equation} 
	Under our assumption there exists $q>\lz$ such that
	\begin{equation}
		M_q<\lz.\nonumber
	\end{equation}
	That is, we can take $r=\lz$ in \eqref{wif}.

From here on we assume
\begin{equation}\label{lzc1}
	\lz<\frac{N+2}{2}.
\end{equation}
Combining this with \eqref{adef} and \eqref{mqd} yields
\begin{equation}
	\mq=\frac{(N+2)(q-\lz)}{2\left(q-\frac{N+2}{2}\right)}>\frac{N+2}{2}>\lz.\label{ml}
\end{equation}
Moreover, we can represent $b$ as
\begin{equation}\label{bd1}
	b=\frac{\alpha q(\mq-j)}{q-\lz}.
\end{equation}

To continue the proof of \eqref{cb}, we consider the function
\begin{equation}
	f(r)=\ln\left(\iqt w^rdxdt\right)\ \ \mbox{for $r>\lz$}.\nonumber
\end{equation}
We can easily verify that $	f(r)$ is a convex function on $(\lz, \infty)$. Indeed, let $\lz<r_1<r_2$ and $\lambda\in [0,1]$. The interpolation inequality asserts
\begin{equation}
	\|w\|_{\lambda r_1+(1-\lambda)r_2, Q_T}\leq 	\|w\|_{ r_1, Q_T}^{\frac{\lambda r_1}{\lambda r_1+(1-\lambda)r_2}}	\|w\|_{r_2, Q_T}^{\frac{(1-\lambda)r_2}{\lambda r_1+(1-\lambda)r_2}}.\nonumber
\end{equation}
Raise both sides to the power of $\lambda r_1+(1-\lambda)r_2$ and then take logarithm  to derive
\begin{equation}
	f(\lambda r_1+(1-\lambda)r_2)\leq \lambda f(r_1)+(1-\lambda)f(r_2).\nonumber
\end{equation}
We can also obtain the convexity of $f$ by computing
\begin{eqnarray}
	f^\prime(r)&=&\frac{\iqt w^r\ln wdxdt}{\iqt w^rdxdt},\nonumber\\
	f^{\prime\prime}(r)&=&\frac{\iqt w^r\ln^2 wdxdt\iqt w^rdxdt-\left(\iqt w^r\ln wdxdt\right)^2}{\left(\iqt w^rdxdt\right)^2}.\nonumber
\end{eqnarray}
Apply H\"{o}lder inequality in the expression for $f^{\prime\prime}(r)$ to derive
\begin{equation}\label{pp}
	f^{\prime\prime}(r)\geq 0.
\end{equation}

Next, we evaluate
\begin{eqnarray}
	\lim_{\ell\ra r}\left(\frac{\|w\|_{\ell,Q_T}}{\|w\|_{r,Q_T}}\right)^{\frac{1}{\ell-r}}&=&	\lim_{\ell\ra r}e^{\frac{\ln \|w\|_{\ell,Q_T}-\ln \|w\|_{r,Q_T}}{\ell-r}}\nonumber\\
	&=&e^{\left(\frac{f(r)}{r}\right)^\prime}\nonumber\\
	&=&\|w\|_{r,Q_T}^{-\frac{1}{r}}e^{\frac{\iqt w^r\ln wdxdt}{r\iqt w^rdxdt}}.\nonumber
\end{eqnarray}
Plug \eqref{bd1} into \eqref{est2} and then take $\ell\ra r$ there to obtain
 \begin{eqnarray}
 	\|w\|_{\infty,Q_T}&\leq& 8I\|w(\cdot,0)\|_{\infty,\rn}\nonumber\\
 	&&+ C\|w\|_{\lz,Q_T}^{s_1}\left[\lim_{\ell\ra r}\left(\frac{\|w\|_{\ell,Q_T}}{\|w\|_{r,Q_T}}\right)^{\frac{\ell}{\ell-r}}\right]^{-\frac{j}{(r-2j)}}\|w\|_{r,Q_T}^{1+\frac{j}{r-2j}}\|w\|_{q,Q_T}^{\frac{q(\mq-j)}{(q-\lz)(r-2j)}}\nonumber\\
 	&\leq& 8I\|w(\cdot,0)\|_{\infty,\rn}+ C\|w\|_{\lz,Q_T}^{s_1}\left[\|w\|_{r,Q_T}^{-1}e^{\fp}\right]^{-\frac{j}{(r-2j)}}\|w\|_{r,Q_T}^{1+\frac{j}{r-2j}}\|w\|_{q,Q_T}^{\frac{q(\mq-j)}{(q-\lz)(r-2j)}}\nonumber\\
 	&=& 8I\|w(\cdot,0)\|_{\infty,\rn}+ C\|w\|_{\lz,Q_T}^{s_1}e^{-\frac{j\fp}{(r-2j)}+\frac{\fr}{r-2j}}\|w\|_{q,Q_T}^{\frac{q(\mq-j)}{(q-\lz)(r-2j)}}	.\label{est10}
 \end{eqnarray}
 Note that $C$ and $s_1$ remain the same because they do not depend on $\ell$.
 
The rest of the proof is similar to that in \cite{X1}. For the reader's convenience, we reproduce it here.
In view of \eqref{mql}, we may pick
\begin{equation}\label{qmq}
	q>2\mq.
\end{equation}
Subsequently, take
\begin{equation}\label{jl}
	j>\frac{q}{2}.
\end{equation}
Introduce a new parameter
\begin{equation}
	\ve>0.\nonumber
\end{equation}
Then
select
\begin{equation}
	r>2j+\ve.\nonumber
\end{equation}
This choice of $r$ satisfies \eqref{rcon}. 
The introduction of $\ve$ is to ensure that $r$ stays away from $2j$ because the constant $C$ and $s_1$ in \eqref{est10} blow up as $r\ra 2j$.
In summary, we have
\begin{equation}\label{rq}
	2\mq<	q<2j<2j+\ve<r.
\end{equation}
We can form the interpolation inequality
\begin{equation}\label{int}
	\|w\|_{2j+\ve,Q_T}\leq \|w\|_{r,Q_T}^{\frac{r(2j+\ve-q)}{(2j+\ve)(r-q)}}\|w\|_{q,Q_T}^{\frac{q(r-2j-\ve)}{(2j+\ve)(r-q)}}
\end{equation}
Note from  \eqref{rq} that
\begin{equation}
	j>\mq.\nonumber
\end{equation}
We may raise both sides of \eqref{int} to the power of $\frac{(j-\mq)(2j+\ve)(r-q)}{(q-\lz)(r-2j)(r-2j-\ve)}$, thereby obtaining
\begin{equation}
	\|w\|_{2j+\ve,Q_T}^{\frac{(j-\mq)(2j+\ve)(r-q)}{(q-\lz)(r-2j)(r-2j-\ve)}}\leq \|w\|_{r,Q_T}^{\frac{r(2j+\ve-q)(j-\mq)}{(q-\lz)(r-2j)(r-2j-\ve)}}\|w\|_{q,Q_T}^{\frac{q(j-\mq)}{(q-\lz)(r-2j)}}.\nonumber
\end{equation}
Incorporating this into \eqref{est10}, we arrive at
\begin{eqnarray}
	\|w\|_{\infty,Q_T}&\leq& 8I\|w(\cdot,0)\|_{\infty,\rn}\nonumber\\
	&&+ C\|w\|_{\lz,Q_T}^{s_1}e^{\frac{-j\fp+\fr}{r-2j}+\frac{(2j+\ve-q)(j-\mq)\fr}{(q-\lz)(r-2j)(r-2j-\ve)}}\|w\|_{2j+\ve,Q_T}^{-\frac{(j-\mq)(2j+\ve)(r-q)}{(q-\lz)(r-2j)(r-2j-\ve)}}.\label{ja2}
\end{eqnarray}
Fix
\begin{equation}
	\eta\in (0,1).\nonumber
\end{equation}
Without any loss of generality, we may assume
\begin{eqnarray}
	\lefteqn{-jf^\prime(s)+f(s)+\frac{(2j+\ve-q)(j-\mq)f(s)}{(q-\lz)(s-2j-\ve)}}\nonumber\\ &\geq&\left[\frac{(j-\mq)(s-q)}{(q-\lz)(s-2j-\ve)}+\frac{(s-2j)(1-\eta)}{2j+\ve-\lz}\right]f(2j+\ve) \ \ \mbox{for each $s\in (2j+\ve, r]$.}\label{pp1}
\end{eqnarray}
Indeed, suppose this is not true. That is, there is a $s\in (2j+\ve, r]$ such that
\begin{eqnarray}
	\lefteqn{-jf^\prime(s)+f(s)+\frac{(2j+\ve-q)(j-\mq)f(s)}{(q-\lz)(s-2j-\ve)}}\nonumber\\ &<&\left[\frac{(j-\mq)(s-q)}{(q-\lz)(s-2j-\ve)}+\frac{(s-2j)(1-\eta)}{2j+\ve-\lz}\right]f(2j+\ve).\nonumber
\end{eqnarray}
Obviously, \eqref{ja2} holds for $r=s$. As a result, we can apply the preceding inequality to it. Upon doing so, we arrive at
\begin{eqnarray}
	\|w\|_{\infty,Q_T}&\leq& 8I\|w(\cdot,0)\|_{\infty,\rn}\nonumber\\
	&&+ C\|w\|_{\lz,Q_T}^{s_1}\|w\|_{2j+\ve,Q_T}^{-\frac{(j-\mq)(2j+\ve)(s-q)}{(q-\lz)(s-2j)(s-2j-\ve)}+\left[\frac{(j-\mq)(s-q)}{(q-\lz)(s-2j-\ve)}+\frac{(s-2j)(1-\eta)}{2j+\ve-\lz}\right]\frac{2j+\ve}{s-2j}}\nonumber\\
	&=& 8I\|w(\cdot,0)\|_{\infty,\rn}+ C\|w\|_{\lz,Q_T}^{s_1}\|w\|_{2j+\ve,Q_T}^{\frac{(2j+\ve)(1-\eta)}{2j+\ve-\lz}}.\label{est11}
\end{eqnarray}
As we noted earlier, $C$ and $s_1$ here remain bounded for $s\in (2j+\ve, r]$.
In view of \eqref{rq}, \eqref{qcon} and \eqref{lzc1}, we can form the interpolation inequality
\begin{equation}
	\|w\|_{2j+\ve,Q_T}\leq \|w\|_{\infty,Q_T}^{\frac{2j+\ve-\lz}{2j+\ve}}\|w\|_{\lz,Q_T}^{\frac{\lz}{2j+\ve}}.\nonumber
\end{equation}
Collect this  in \eqref{est11} to get
\begin{eqnarray}
	\|w\|_{\infty,Q_T}&\leq& 8I\|w(\cdot,0)\|_{\infty,\rn}+C\|w\|_{\lz,Q_T}^{s_1+\frac{\lz(1-\eta)}{2j+\ve-\lz}}\|w\|_{\infty,Q_T}^{1-\eta}.\nonumber
\end{eqnarray}
Then \eqref{cb} follows from a suitable application of Young's inequality (\cite{GT}, p. 145). That is to say, if \eqref{pp1} fails to be true, then \eqref{cb} holds.

We next show that \eqref{cb} remains valid under \eqref{pp1}. On account of \eqref{pp},  the convexity of $f(s)$, there holds
\begin{equation}\label{ja4}
	f^\prime(s)\geq	\frac{f(s)-f(2j+\ve)}{s-2j-\ve}\ \ \mbox{for each $s>2j+\ve$}.
\end{equation}
We may decompose
\begin{equation}
	\frac{(j-\mq)(s-q)}{(q-\lz)(s-2j-\ve)}=\frac{(j-\mq)}{(q-\lz)}+\frac{(j-\mq)(2j+\ve-q)}{(q-\lz)(s-2j-\ve)}.\nonumber
\end{equation}
Incorporate the preceding two results into \eqref{pp1} to derive
\begin{eqnarray}
	-s_jf^\prime(s)+f(s) &\geq&\left[\frac{(j-\mq)}{(q-\lz)}+\frac{(s-2j)(1-\eta)}{2j+\ve-\lz}\right]f(2j+\ve) \ \ \mbox{for each $s\in (2j+\ve, r]$.}\label{ja3}
\end{eqnarray}
where
\begin{equation}\label{sj}
	s_j= j-	\frac{(2j+\ve-q)(j-\mq)}{(q-\lz)}.
\end{equation}
We further require 
\begin{equation}\label{ja8}
	s_j>0.
\end{equation}
For the  inequality to hold, it is enough for us to take
\begin{equation}\label{eu}
	\ve<\frac{q-\lz}{j-\mq}\left[j-\frac{(2j-q)(j-\mq)}{q-\lz}\right].
\end{equation}
This is possible only when
\begin{equation}\label{ja7}
	\frac{(2j-q)(j-\mq)}{q-\lz}-	j=\frac{2j^2-(2q+2\mq-\lz)j+\mq q}{q-\lz}<0.
\end{equation}
The numerator is a quadratic function in $j$. The discriminant is given by
\begin{eqnarray}
	(2q+2\mq-\lz)^2-8\mq q&=&4q^2+4\mq^2-4\lz(q+\mq)+\lz^2\nonumber\\
	&=&(2q-\lz)^2+4\mq(\mq-\lz)>0.\nonumber
\end{eqnarray} 
The last step here is due to \eqref{ml}. According to the quadratic formula, the solution set to \eqref{ja7} is given by
\begin{equation}
	j\in (j_1, j_2),\nonumber
\end{equation}
where
\begin{eqnarray}
	j_1&=&\frac{2q+2\mq-\lz-\sqrt{(2q+2\mq-\lz)^2-8\mq q}}{4},\nonumber\\
	j_1&=&\frac{2q+2\mq-\lz+\sqrt{(2q+2\mq-\lz)^2-8\mq q}}{4}.\nonumber
\end{eqnarray}
Obviously, both $j_1$ and $j_2$ are positive. In particular,
\begin{equation}\label{j2}
	j_2\in (q, q+\mq-\lz).
\end{equation}
Under \eqref{ja8}, we can write \eqref{ja3} in the form
\begin{eqnarray}
	f^\prime(s)-\frac{f(s)}{s_j} &\leq&-\left[\frac{(j-\mq)}{(q-\lz)}+\frac{(s-2j)(1-\eta)}{2j+\ve-\lz}\right]\frac{f(2j+\ve)}{s_j} \ \ \mbox{for each $s\in (2j+\ve, r]$.}\nonumber
\end{eqnarray}
Multiply through the inequality by $e^{-\frac{s}{s_j}}$ to derive 
\begin{equation}
	\left(e^{-\frac{s}{s_j}}f(s)\right)^\prime\leq -\left[\frac{(j-\mq)}{(q-\lz)}+\frac{(s-2j)(1-\eta)}{2j+\ve-\lz}\right]\frac{f(2j+\ve)e^{-\frac{s}{s_j}}}{s_j} \ \ \mbox{for each $s\in (2j+\ve, r]$.}\nonumber
\end{equation}
Subsequently,
\begin{eqnarray}
	e^{-\frac{r}{s_j}}	f(r)&=&e^{-\frac{2j+\ve}{s_j}}f(2j+\ve)+\int_{2j+\ve}^{r}	\left(e^{-\frac{s}{s_j}}f(s)\right)^\prime ds\nonumber\\
	&\leq&e^{-\frac{2j+\ve}{s_j}}f(2j+\ve)+\frac{(j-\mq)\left(e^{-\frac{r}{s_j}}-e^{-\frac{2j+\ve}{s_j}}\right)f(2j+\ve)}{(q-\lz)}\nonumber\\
	&&+\frac{(1-\eta)f(2j+\ve)}{2j+\ve-\lz}\left[(r-2j)e^{-\frac{r}{s_j}}-\ve e^{-\frac{2j+\ve}{s_j}}+s_j\left(e^{-\frac{r}{s_j}}-e^{-\frac{2j+\ve}{s_j}}\right)\right]\nonumber\\
	&=&\left[1-\frac{j-\mq}{q-\lz}-\frac{(1-\eta)(\ve+s_j)}{2j+\ve-\lz}\right]e^{-\frac{2j+\ve}{s_j}}f(2j+\ve)\nonumber\\
	&&+\left[\frac{j-\mq}{q-\lz}+\frac{(1-\eta)(r-2j+s_j)}{2j+\ve-\lz}\right]e^{-\frac{r}{s_j}}f(2j+\ve).\label{h1}
\end{eqnarray}
Recall \eqref{ja4} and \eqref{sj} to obtain
\begin{eqnarray}
	\lefteqn{-j\fp+\fr+\frac{(2j+\ve-q)(j-\mq)\fr}{(q-\lz)(r-2j-\ve)}}\nonumber\\
	&\leq&-\frac{j(f(r)-f(2j+\ve))}{r-2j-\ve}+\fr+\frac{(2j+\ve-q)(j-\mq)f(r)}{(q-\lz)(r-2j-\ve)}\nonumber\\
	&=&\frac{(r-2j-\ve-s_j)\fr}{r-2j-\ve}+\frac{jf(2j+\ve)}{r-2j-\ve}.\label{h2}
\end{eqnarray}
We further require
\begin{equation}\label{rl}
	r>2j+\ve+s_j.
\end{equation}
Under this choice for $r$, we can combine \eqref{h1} and \eqref{h2} to deduce
\begin{eqnarray}
	\lefteqn{-j\fp+\fr+\frac{(2j+\ve-q)(j-\mq)\fr}{(q-\lz)(r-2j-\ve)}}\nonumber\\
	&\leq&\frac{(r-2j-\ve-s_j)}{r-2j-\ve}\left[1-\frac{j-\mq}{q-\lz}-\frac{(1-\eta)(\ve+s_j)}{2j+\ve-\lz}\right]e^{\frac{r-2j-\ve}{s_j}}f(2j+\ve)\nonumber\\
	&&+\frac{(r-2j-\ve-s_j)}{r-2j-\ve}\left[\frac{j-\mq}{q-\lz}+\frac{(1-\eta)(r-2j+s_j)}{2j+\ve-\lz}\right]f(2j+\ve)+\frac{jf(2j+\ve)}{r-2j-\ve}.\nonumber
\end{eqnarray}
Utilizing this in \eqref{ja2}, we arrive at
\begin{eqnarray}
	\|w\|_{\infty,Q_T}&\leq&	8I\|w(\cdot,0)\|_{\infty,\rn}+ C\|w\|_{\lz,Q_T}^{s_1}\|w\|_{2j+\ve,Q_T}^{\beta_2},\label{est12}
\end{eqnarray}
where
\begin{eqnarray}
	\beta_2&=&-\frac{(j-\mq)(2j+\ve)(r-q)}{(q-\lz)(r-2j)(r-2j-\ve)}+\frac{j(2j+\ve)}{(r-2j-\ve)(r-2j)}\nonumber\\
	&&+\frac{(r-2j-\ve-s_j)e^{\frac{r-2j-\ve}{s_j}}(2j+\ve)}{(r-2j-\ve)(r-2j)}\left[1-\frac{j-\mq}{q-\lz}-\frac{(1-\eta)(\ve+s_j)}{2j+\ve-\lz}\right]\nonumber\\
	&&+\frac{(r-2j-\ve-s_j)(2j+\ve)}{(r-2j-\ve)(r-2j)}\left[\frac{j-\mq}{q-\lz}+\frac{(1-\eta)(r-2j+s_j)}{2j+\ve-\lz}\right].\label{b2}
\end{eqnarray}
We must have
\begin{equation}\label{dh6}
	\beta_2<\frac{2j+\ve}{2j+\ve-\lz}.
\end{equation}
Plug \eqref{b2} into this and simplify the resulting inequality to obtain
\begin{eqnarray}
	\lefteqn{-\frac{(j-\mq)(r-\lz)}{(q-\lz)}+r-\mq}\nonumber\\
	&&+(r-2j-\ve-s_j)e^{\frac{r-2j-\ve}{s_j}}\left[1-\frac{j-\mq}{q-\lz}-\frac{(1-\eta)(\ve+s_j)}{2j+\ve-\lz}\right]\nonumber\\
	&&+(r-2j-\ve-s_j)\left[\frac{j-\mq}{q-\lz}+\frac{(1-\eta)(r-2j+s_j)}{2j+\ve-\lz}\right]<\frac{(r-2j)(r-\lz)}{2j+\ve-\lz},\nonumber
\end{eqnarray}
from whence it follows
\begin{eqnarray}
	\lefteqn{\left[(\ve+s_j)e^{\frac{r-2j-\ve}{s_j}}-(r-2j+s_j)\right]\frac{(r-2j-\ve-s_j)\eta}{2j+\ve-\lz}}	\nonumber\\
	&<&\frac{(r-2j)(r-\lz)}{2j+\ve-\lz}+\frac{(j-\mq)(r-\lz)}{(q-\lz)}-r+\mq\nonumber\\
	&&+(r-2j-\ve-s_j)\left[e^{\frac{r-2j-\ve}{s_j}}\left[\frac{j-\mq}{q-\lz}+\frac{(\ve+s_j)}{2j+\ve-\lz}-1\right]-\frac{j-\mq}{q-\lz}-\frac{(r-2j+s_j)}{2j+\ve-\lz}\right]\nonumber\\
	&\equiv& h(r).\label{dh3}
\end{eqnarray}
We easily ckeck that
\begin{equation}
	\frac{(\ve+s_j)e^{\frac{r-2j-\ve}{s_j}}}{2j+\ve-\lz}-\frac{(r-2j+s_j)}{2j+\ve-\lz}>0\ \ \mbox{for $r>2j+\ve$.}\nonumber
\end{equation}
Indeed, there hold
\begin{eqnarray}
	\left.(\ve+s_j)e^{\frac{r-2j-\ve}{s_j}}-(r-2j+s_j)	\right|_{r=2j+\ve}&=&0,\nonumber\\
	\left[(\ve+s_j)e^{\frac{r-2j-\ve}{s_j}}-(r-2j+s_j)\right]^\prime&>&0\ \ \mbox{for $r>2j+\ve$.}\nonumber
\end{eqnarray}
Hence, we must show that there exists a $r>2j+\ve+s_j$ such that
\begin{eqnarray}
	h(r)>0 .\label{dh4}
\end{eqnarray}
Note from \eqref{sj} that
\begin{eqnarray}
	\frac{r-2j+s_j}{2j+\ve-\lz}&=&\frac{r-j-\frac{(2j+\ve-q)(j-\mq)}{q-\lz}}{2j+\ve-\lz}\nonumber\\
	&=&\frac{r-\mq-\frac{(2j+\ve-\lz)(j-\mq)}{q-\lz}}{2j+\ve-\lz}\nonumber\\
	&=&\frac{r-\mq}{2j+\ve-\lz}-\frac{(j-\mq)}{q-\lz}.\nonumber
\end{eqnarray}
By the same token,
\begin{eqnarray}
	\frac{2j-s_j-\mq}{2j+\ve-\lz}&=&\frac{j+\frac{(2j+\ve-q)(j-\mq)}{q-\lz}-\mq}{2j+\ve-\lz}\nonumber\\
	&=&\frac{(j-\mq)}{q-\lz},\nonumber\\
	\frac{s_j+\ve}{2j+\ve-\lz}&=&\frac{j-\frac{(2j+\ve-q)(j-\mq)}{q-\lz}+\ve}{2j+\ve-\lz}\nonumber\\
	&=&\frac{2j-\mq-\frac{(2j+\ve-\lz)(j-\mq)}{q-\lz}+\ve}{2j+\ve-\lz}\nonumber\\
	&=&\frac{2j-\mq+\ve}{2j+\ve-\lz}-\frac{j-\mq}{q-\lz}.\nonumber
\end{eqnarray}
Incorporate the preceding three equations into the expression for $h(r)$ in \eqref{dh3} to derive
\begin{eqnarray}
	h(r)&=&\frac{(r-2j)(r-\lz)}{2j+\ve-\lz}+\frac{(j-\mq)(r-\lz)}{(q-\lz)}-r+\mq\nonumber\\
	&&+(r-2j-\ve-s_j)\left[e^{\frac{r-2j-\ve}{s_j}}\left[\frac{2j-\mq+\ve}{2j+\ve-\lz}-1\right]-\frac{r-\mq}{2j+\ve-\lz}\right]\nonumber\\
	&=&\frac{(r-2j)(r-\lz)}{2j+\ve-\lz}+\frac{(j-\mq)(r-\lz)}{(q-\lz)}-\frac{(r-\mq)(r-s_j-\lz)}{2j+\ve-\lz}\nonumber\\
	&&-\frac{(r-2j-\ve-s_j)(\mq-\lz)e^{\frac{r-2j-\ve}{s_j}}}{2j+\ve-\lz}\nonumber\\
	&=&\frac{-(2j-s_j-\mq)r+\lz(2j-\mq)}{2j+\ve-\lz}+\frac{(j-\mq)(r-\lz)}{(q-\lz)}\nonumber\\
	&&-\frac{(r-2j-\ve-s_j)(\mq-\lz)e^{\frac{r-2j-\ve}{s_j}}}{2j+\ve-\lz}\nonumber\\
	&=&\frac{\lz(2j-\mq)}{2j+\ve-\lz}-\frac{\lz(j-\mq)}{(q-\lz)}-\frac{(r-2j-\ve-s_j)(\mq-\lz)e^{\frac{r-2j-\ve}{s_j}}}{2j+\ve-\lz}.\nonumber
\end{eqnarray}
In view of \eqref{ml}, for \eqref{dh4} to hold, we must have
\begin{equation}
	\frac{(2j-\mq)}{2j+\ve-\lz}-\frac{(j-\mq)}{(q-\lz)}>0.\nonumber
\end{equation}
Solve this for $\ve$ to deduce
\begin{eqnarray}
	\ve&<&\frac{q-\lz}{j-\mq}\left[2j-\mq-\frac{(j-\mq)(2j-\lz)}{q-\lz}\right]\nonumber\\
	&=&\frac{q-\lz}{j-\mq}\left[j-\frac{(j-\mq)(2j-q)}{q-\lz}\right],\nonumber
\end{eqnarray}
which is exactly our assumption \eqref{eu}.

In summary, the order in which we choose our parameters is as follows: Let $q$ be given as in \eqref{qmq}. By virtue of  \eqref{j2}, we may take
\begin{equation}
	j\in\left(\max\left\{\frac{q}{2},j_1\right\}, j_2\right).\nonumber
\end{equation}
This implies \eqref{jl} and  enables us to select $\ve$ as in \eqref{eu},  which, in turn,  guarantees \eqref{dh4} for  $r$ close to $2j+\ve+s_j$ from the right-hand side. That is, we have both \eqref{rl} and \eqref{dh4}. Equipped with this, we can choose $\eta\in(0,1)$ so that \eqref{dh3} is satisfied.

Without any loss of generality, we may assume
\begin{equation}\label{h4}
	\|w\|_{2j+\ve,Q_T}>1. 
\end{equation}
Otherwise, \eqref{wif} would be enough to imply our theorem. Under \eqref{h4}, we may suppose that the last exponent $\beta_2$ in \eqref{est12} is positive.
Were this not true, \eqref{h4} combined with \eqref{est12} would yield our theorem. 
In view of \eqref{qcon} and \eqref{lzc1}, we can form the interpolation inequality
\begin{equation}
	\|w\|_{2j+\ve,Q_T}\leq \|w\|_{\infty,Q_T}^{\frac{2j+\ve-\lz}{2j+\ve}}\|w\|_{\lz,Q_T}^{\frac{\lz}{2j+\ve}}.\nonumber
\end{equation}
Collect this  in \eqref{est12} and keep in mind that $\beta_2>0$ to get
\begin{eqnarray}
	\|w\|_{\infty,Q_T}&\leq& 8I\|w(\cdot,0)\|_{\infty,\rn}+C\|w\|_{\lz,Q_T}^{s_1+\frac{\lz\beta_2}{2j+\ve}}\|w\|_{\infty,Q_T}^{\frac{(2j+\ve-\lz)\beta_2}{2j+\ve}},\nonumber
\end{eqnarray}
According to \eqref{dh6}, we have
\begin{equation}
	\frac{(2j+\ve-\lz)\beta_2}{2j+\ve}\in (0,1).\nonumber
\end{equation}
Then \eqref{cb} follows from a suitable application of Young's inequality.
The proof of Theorem \ref{thm} is completed.
\end{proof}

\begin{proof}[Proof of \eqref{ub}]

 Set
	\begin{equation*}
	A_r=\||u|^2\|_{r,Q_T}\ \ \mbox{	for $r>1$}. 	
	\end{equation*}
	Subsequently, let
	\begin{equation}\label{aw1}
		\vp=\frac{|u|^2}{A_r}.
	\end{equation}
	We proceed to derive an equation for $\vp$ as in \cite{X}. To this end,
	we take the dot product of both sides of \eqref{nsf1} with $u$ to obtain
	\begin{equation}\label{ns10}
	\pt u\cdot u+ (u\cdot\nabla) u\cdot u+\nabla p\cdot u=\Delta u\cdot u-\Psi\nabla\phi\cdot u.
	\end{equation}
	We calculate from \eqref{nsf4} that
	\begin{eqnarray*}
		u\cdot\pt u&=&\frac{1}{2}\pt |u|^2,\\
		(u\cdot\nabla) u\cdot u&=&u_j\pxj u_iu_i=\frac{1}{2}(u\cdot\nabla)|u|^2,\\
		\nabla p\cdot u&=&\nabla\cdot(pu),\\
		\Delta u\cdot u&=&	\frac{1}{2}\Delta|u|^2-|\nabla u|^2.
	\end{eqnarray*}
	Here we have employed the notation convention of summing over repeated indices.
	Substitute the preceding four equations into \eqref{ns10} and divide through the resulting equation by $\frac{A_r}{2}$ to derive
	\begin{equation}\label{use}
	\pt \vp+(u\cdot\nabla)\vp-\Delta\vp+ 2A_r^{-1}|\nabla u|^2=-2A_r^{-1}\nabla\cdot(pu)-2A_r^{-1}\Psi\nabla\phi\cdot u.	
	\end{equation}
We are in a position to employ the previous De Giorgi iteration scheme. Let
\begin{equation*}
	k\geq 2\|\vp(\cdot,0)\|_{\infty,\rn}.
\end{equation*}
Define $k_n$ as before. Use
$$\left(\ln{\vp}-\ln{k_{n}}\right)^+$$
as a test function in \eqref{use} to derive
\begin{eqnarray}
	\lefteqn{\frac{d}{dt}\irn \int_{k_n}^{\vp}\left(\ln\mu-\ln k_{n}\right)^+d\mu dx+\int_{\{\vp(\cdot,t)\geq k_{n}\}}\frac{1}{\vp}|\nabla\vp|^2dx}\nonumber\\
	&\leq&-\irn (u\cdot\nabla)\vp\left(\ln{\vp}-\ln{k_{n}}\right)^+dx+2A_r^{-1} \int_{\{\vp(\cdot,t)\geq k_{n}\}}\frac{p}{\vp} (u\cdot\nabla) \vp dx\nonumber\\
	&&-2A_r^{-1} \irn\Psi \nabla\phi\cdot u \left(\ln{\vp}-\ln{k_{n}}\right)^+dx.\label{pub1}
\end{eqnarray}
We proceed to analyze each term in the above inequality.
First, note from \eqref{nsf4} that
\begin{eqnarray*}
-\irn (u\cdot\nabla)\vp\left(\ln{\vp}-\ln{k_{n}}\right)^+dx=-\irn (u\cdot\nabla)\vp\left(\ln{\vp}-\ln{k_{n}}\right)^+dx.
\end{eqnarray*}
In view of \eqref{aw1}, we have
\begin{eqnarray}
	2A_r^{-1} \int_{\{\vp(\cdot,t)\geq k_{n}\}}\frac{p}{\vp} (u\cdot\nabla) \vp dx&\leq&2A_r^{-\frac{1}{2}}\int_{\{\vp(\cdot,t)\geq k_{n}\}}\frac{1}{\vp^{\frac{1}{2}}}|p| |\nabla \vp| dx\nonumber\\
	&\leq&\frac{1}{2}\int_{\{\vp(\cdot,t)\geq k_{n}\}}\frac{1}{\vp}|\nabla\vp|^2dx+2A_r^{-1}\int_{\{\vp(\cdot,t)\geq k_{n}\}}|p|^2dx.\nonumber
\end{eqnarray}
Remember from \eqref{nsfa} that
\begin{equation*}
	|\Psi|\leq \max_{1\leq i\leq I} |z_i|w\leq c.
\end{equation*}
In view of \eqref{nsf3}, we may apply the classical representation theorem (\cite{GT}, p. 17) to obtain
\begin{equation}\label{pre}
	\phi(x,t)=\irn\Gamma(y-x)\Psi(y,t) dy,
\end{equation}
where $\Gamma (x)$ is the fundamental solution of the Laplace equation, i.e.,
$$\Gamma(x)=\frac{1}{N(N-2)\omega_N |x|^{N-2}},\ \ \omega_N= \mbox{the volume of the unit ball in $\rn$.}$$
It immediately follows that
\begin{equation*}
	|\phi|\leq c\irn\frac{w}{|x-y|^{N-2}}dy,\ \ 	|\nabla\phi|\leq c\irn\frac{w}{|x-y|^{N-1}}dy.
\end{equation*}
This together with Theorem 1 in (\cite{S}, p. 119) implies
\begin{eqnarray}
	\|\nabla\phi\|_{\frac{Ns}{N-s},\rn}&\leq& c\left\|w\right\|_{s,\rn}\ \ \mbox{for $s\in(1,N)$.} \nonumber
\end{eqnarray}
Hence, for each $s>\frac{N}{N-1}$ there holds
\begin{equation}\label{npe}
	\|\nabla\phi\|_{s,\rn}\leq c\|w\|_{\frac{Ns}{N+s},\rn}\leq c\|w\|_{\infty,\rn}^{1-\frac{N+s}{Ns}}\|w\|_{1,\rn}^\frac{N+s}{Ns}\leq c.
\end{equation}
The last term in \eqref{pub1} can be estimated as follows:
\begin{eqnarray*}
	\lefteqn{-2A_r^{-1} \irn\Psi \nabla\phi\cdot u \left(\ln{\vp}-\ln{k_{n}}\right)^+dx}\nonumber\\
	&\leq&cA_r^{-\frac{1}{2}}\irn |\nabla\phi|\sqrt{\vp}\left(\ln{\vp}-\ln{k_{n}}\right)^+dx\nonumber\\
	&=&cA_r^{-\frac{1}{2}}\irn |\nabla\phi|\left(\sqrt{\vp}-\sqrt{k_n}\right)^+\left(\ln{\vp}-\ln{k_{n}}\right)^+dx\nonumber\\
	&&+cA_r^{-\frac{1}{2}}\sqrt{k_n}\irn |\nabla\phi|\left(\ln{\vp}-\ln{k_{n}}\right)^+dx\equiv I_1+I_2.
\end{eqnarray*}
Fix 
$$s>\frac{N}{2}.$$
Subsequently,
$$\frac{s}{s-1}<\frac{2N}{N-2}.$$
Also, it is easy to check from \eqref{aw1} that
$$\|\vp\|_{r,Q_T}=1.$$
Therefore,
$$\left|\{\vp\geq k_{n}\}\right|\leq \irnt\left(\frac{2\vp}{k}\right)^rdxdt\leq 2^rk^{-r}.$$
With these in mind, we calculate from  \eqref{sob} and \eqref{npe} that
\begin{eqnarray}
	I_1	&\leq&cA_r^{-\frac{1}{2}}\|\nabla\phi\|_{s,\rn}\left\|\left(\sqrt{\vp}-\sqrt{k_n}\right)^+\right\|_{\frac{2N}{N-2},\rn}\left\|\left(\ln{\vp}-\ln{k_{n}}\right)^+\right\|_{\frac{2Ns}{(N+2)s-2N},\rn}\nonumber\\
		&\leq&cA_r^{-\frac{1}{2}}\left(\int_{\{\vp(\cdot,t)\geq k_{n}\}}\frac{1}{\vp}|\nabla\vp|^2dx\right)^{\frac{1}{2}}\left\|\left(\ln{\vp}-\ln{k_{n}}\right)^+\right\|_{\frac{2N}{N-2},\rn}\left|\{\vp(\cdot,t)\geq k_{n}\}\right|^{\frac{(N+2)s-2N}{2Ns}-\frac{N-2}{2N}}\nonumber\\
	&\leq&cA_r^{-\frac{1}{2}}\left|\{\vp(\cdot,t)\geq k_{n}\}\right|^{\frac{2s-N}{Ns}}\left(\int_{\{\vp(\cdot,t)\geq k_{n}\}}\frac{1}{\vp}|\nabla\vp|^2dx\right)^{\frac{1}{2}}\left(\int_{\{\vp(\cdot,t)\geq k_{n}\}}\frac{1}{\vp^2}|\nabla\vp|^2dx\right)^{\frac{1}{2}}\nonumber\\
	&\leq&cA_r^{-\frac{1}{2}}k^{-\frac{1}{2}-\frac{r(2s-N)}{Ns}}\int_{\{\vp(\cdot,t)\geq k_{n}\}}\frac{1}{\vp}|\nabla\vp|^2dx	.\label{nes1}
\end{eqnarray}
We choose $k$ so large that the coefficient of the last integral in \eqref{nes1} is less than $\frac{1}{8}$, i.e.,
\begin{eqnarray}
\frac{1}{8}	k^{\frac{1}{2}+\frac{r(2s-N)}{Ns}}\geq cA_r^{-\frac{1}{2}}.\nonumber
\end{eqnarray}
Consequently, 
\begin{eqnarray*}
	I_1\leq \frac{1}{8}\int_{\{\vp(\cdot,t)\geq k_{n}\}}\frac{1}{\vp}|\nabla\vp|^{2}dx.
\end{eqnarray*}
Similarly, 
\begin{eqnarray*}
I_2
	&=&cA_r^{-\frac{1}{2}}\sqrt{k_n}\irn |\nabla\phi|\left(\ln{\vp}-\ln{k_{n}}\right)^+dx\nonumber\\
	&\leq&cA_r^{-\frac{1}{2}}\sqrt{k_n}\left(\irn|\nabla\phi|^sdx\right)^{\frac{1}{s}}\left(\irn\left[\left(\ln{\vp}-\ln{k_{n}}\right)^+\right]^{\frac{s}{s-1}}dx\right)^{\frac{s-1}{s}}\nonumber\\
	&\leq&cA_r^{-\frac{1}{2}}\sqrt{k_n}\left(\irn\left[\left(\ln{\vp}-\ln{k_{n}}\right)^+\right]^{\frac{2N}{N-2}}dx\right)^{\frac{N-2}{2N}}\left|\{\vp(\cdot,t)\geq k_{n}\}\right|^{\frac{s-1}{s}-\frac{N-2}{2N}}\nonumber\\
			&\leq&cA_r^{-\frac{1}{2}}\sqrt{k_n}\left(\int_{\{\vp(\cdot,t)\geq k_{n}\}}\frac{1}{\vp^2}|\nabla\vp|^{2}dx\right)^{\frac{1}{2}}\left|\{\vp(\cdot,t)\geq k_{n}\}\right|^{\frac{(N+2)s-2N}{2Ns}}\nonumber\\
			&\leq&cA_r^{-\frac{1}{2}}\left(\int_{\{\vp(\cdot,t)\geq k_{n}\}}\frac{1}{\vp}|\nabla\vp|^{2}dx\right)^{\frac{1}{2}}\left|\{\vp(\cdot,t)\geq k_{n}\}\right|^{\frac{(N+2)s-2N}{2Ns}}\nonumber\\
		&\leq&\frac{1}{16}\int_{\{\vp(\cdot,t)\geq k_{n}\}}\frac{1}{\vp}|\nabla\vp|^{2}dx+cA_r^{-1}\left|\{\vp(\cdot,t)\geq k_{n}\}\right|^{\frac{(N+2)s-2N}{Ns}}	.
\end{eqnarray*}
Collecting all the preceding results in \eqref{pub1} yields
\begin{eqnarray}
	\lefteqn{	\frac{d}{dt}\irn\int_{k_{n}}^{\vp}\left(\ln{\mu}-\ln{k_{n}}\right)^+d\mu   dx+\frac{1}{16}\int_{\{\vp(\cdot,t)\geq k_{n}\}}\frac{1}{\vp}\left|\nabla \vp\right|^2 dx}\nonumber\\
	&\leq&
	cA_r^{-1}\int_{\{\vp(\cdot,t)\geq k_{n}\}} p^2dx+cA_r^{-1}\left|\{\vp(\cdot,t)\geq k_{n}\}\right|^{\frac{(N+2)s-2N}{Ns}}
	.\label{nes2}
\end{eqnarray}
We easily see that
\begin{equation*}
	\int_{\{\vp(\cdot,t)\geq k_{n}\}}\frac{1}{\vp}\left|\nabla \vp\right|^2 dx=4\irn\left|\nabla\left(\sqrt{\vp}-\sqrt{k_n}\right)^+\right|^2dx.
\end{equation*}
We can also infer from the proof of \eqref{hope10} that
\begin{equation*}
	\int_{k_{n}}^{\vp}\left(\ln{\mu}-\ln{k_{n}}\right)^+d\mu\geq 2\left[\left(\sqrt{\vp}-\sqrt{k_n}\right)^+\right]^2.
\end{equation*}
Recall \eqref{4kcon3} and \eqref{kcon1} to derive
\begin{equation*}
	\left.\int_{k_{n}}^{\vp }\left(\ln{\mu}-\ln{k_{n}}\right)^+d\mu\right|_{t=0}=0.
\end{equation*}
Equipped with these estimates, we integrate \eqref{nes2} with respect to $t$ to deduce
\begin{eqnarray}
\lefteqn{	\sup_{0\leq t\leq T}\irn\left[\left(\sqrt{\vp}-\sqrt{k_n}\right)^+\right]^2dx	+\irnt\left|\nabla\left(\sqrt{\vp}-\sqrt{k_n}\right)^+\right|^2dxdt}\nonumber\\
	&\leq &cA_r^{-1}\int_{\{\vp\geq k_{n}\}} p^2dxdt+cA_r^{-1}\left|\{\vp\geq k_{n}\}\right|\equiv I.\nonumber
\end{eqnarray}
Here we have taken $s=N$.
Now set
\begin{equation*}
	y_n=|\{\vp\geq k_{n}\}|.
\end{equation*}
%
We proceed to show that $\{y_n\}$ satisfies the condition in Lemma \ref{ynb}. 
By calculations similar to those in \eqref{ns4},  we have
\begin{eqnarray}
	\lefteqn{	\irnt\left[\left(\sqrt{\vp}-\sqrt{k_n}\right)^+\right]^{\frac{4}{N}+2}dxdt}\nonumber\\
	&\leq&\int_{0}^{T}\left(\irn\left[\left(\sqrt{\vp}-\sqrt{k_n}\right)^+\right]^{2}dx \right)^{\frac{2}{N}}\left(\irn\left[\left(\sqrt{\vp}-\sqrt{k_n}\right)^+\right]^{\frac{2N}{N-2}}dx\right)^{\frac{N-2}{N}}dt\nonumber\\
	&\leq& c\left(\sup_{0\leq t\leq T}\irn\left[\left(\sqrt{\vp}-\sqrt{k_n}\right)^+\right]^{2}dx \right)^{\frac{2}{N}}\irnt\left|\nabla\left(\sqrt{\vp}-\sqrt{k_n}\right)^+\right|^2dxdt\nonumber\\
	&\leq& cI^{\frac{N+2}{N}}.\label{prub3}
\end{eqnarray}
It is easy to verify that
\begin{eqnarray*}
	\irnt\left[\left(\sqrt{\vp}-\sqrt{k_n}\right)^+\right]^{\frac{4}{N}+2}dxdt&\geq&
	\int_{\{\vp\geq k_{n+1}\}}\left[\left(\sqrt{\vp}-\sqrt{k_n}\right)^+\right]^{\frac{4}{N}+2}dxdt\nonumber\\
	&\geq&\left(\sqrt{k_{n+1}}-\sqrt{k_n}\right)^{\frac{4}{N}+2}|\{\vp\geq k_{n+1}\}|
	\nonumber\\
	&\geq&c\left(\frac{\sqrt{k}}{2^{n+2}}\right)^{\frac{4}{N}+2}|\{\vp\geq k_{n+1}\}|
	\nonumber\\
	&	\geq&\frac{c|\{\vp\geq k_{n+1}\}|k^{\frac{(N+2)}{N}}}{2^{(\frac{4}{N}+2)n}}.
\end{eqnarray*}
Combining this with \eqref{prub3} yields
\begin{eqnarray}
	y_{n+1}&=&|\{\vp\geq k_{n+1}\}|^{\frac{N}{N+2}+\frac{2}{N+2}}\nonumber\\
	&\leq& c\frac{4^{n}I}{k}|\{\vp\geq k_{n+1}\}|^{\frac{2}{N+2}}\nonumber\\
	&=& \frac{c4^{n}}{A_rk}\left(\int_{\{\vp\geq k_{n}\}} p^2dxdt+\left|\{\vp\geq k_{n}\}\right|\right)\{\vp\geq k_{n+1}\}|^{\frac{2}{N+2}}.\label{prub4}
\end{eqnarray}

Now we turn our attention to $p$. Take the divergence of both sides of \eqref{nsf1} to obtain
\begin{equation*}
	-	\Delta p=\nabla\cdot (\Psi\nabla\phi)+\nabla\cdot((u\cdot\nabla) u).
\end{equation*}
As in \eqref{pre}, we can also represent $p$ as 
\begin{equation*}
	p(x,t)=\irn \Gamma(y-x)\left[\nabla\cdot (\Psi\nabla\phi)+\nabla\cdot((u\cdot\nabla) u)\right]dy.
\end{equation*} 
We observe from \eqref{nsf4} that
\begin{equation*}
	\irn \Gamma(y-x)\nabla\cdot((u\cdot\nabla) u)dy=\irn \Gamma_{y_iy_j}(y-x)u_iu_j dy.
\end{equation*}
It is a well known fact that $\partial^2_{y_iy_j}\Gamma(y)$ is a Calder\'{o}n-Zygmund kernel. A result of \cite{CFL} asserts that for each $ \ell \in (1,\infty)$ there is a positive number $c_\ell$ determined by $N$ and $\ell $ such that 
\begin{equation*}
	\left\|\irn \Gamma_{y_iy_j}(y-x)u_iu_j dy\right\|_{\ell,\rn}\leq c_\ell\||u|^2\|_{\ell,\rn}.
\end{equation*}
Combing this with Theorem 1 in (\cite{S}, p.119) yields
\begin{equation}\label{use1}
	\|p\|_{\ell,\rn}\leq c\|\Psi\nabla\phi\|_{\frac{N\ell}{N+\ell},\rn}+c\||u|^2\|_{\ell,\rn}\ \ \mbox{for each $\ell>\frac{N}{N-1}$}.
\end{equation}
Observe from \eqref{cb} and \eqref{pls} that
\begin{equation*}
	\|\Psi\nabla\phi\|_{\frac{N\ell}{N+\ell},\rn}\leq c, \ \ 
\end{equation*}
provided that
\begin{equation*}
\ell>\frac{N}{N-2}.	
\end{equation*}
As before, we pick 
\begin{equation*}
	q>N+2.
\end{equation*}
Subsequently, by \eqref{use1}, we have
\begin{equation*}
	\int_{\{\vp\geq k_{n}\}}p^2dxdt\leq \|p\|_{q,Q_T}^2|\{\vp\geq k_{n}\}|^{1-\frac{2}{q}}\leq c(1+\|u\|_{2q,Q_T}^4)|\{\vp\geq k_{n}\}|^{1-\frac{2}{q}}.
\end{equation*}
Substitute this into \eqref{prub4} to get
\begin{equation}\label{hap1}
	y_{n+1}\leq \frac{c4^{n}}{A_rk}\left(1+\|u\|_{2q,Q_T}^4+\left|\{\vp\geq k_{n}\}\right|^{\frac{2}{q}}\right)\{\vp\geq k_{n+1}\}|^{1-\frac{2}{q}+\frac{2}{N+2}}.
\end{equation}
Obviously, we may assume that
\begin{equation}\label{hap}
\|u\|_{2q,Q_T}^4\geq 1+\left|\left\{\vp\geq \frac{k}{2}\right\}\right|^{\frac{2}{q}}.	
\end{equation}
Consequently,
\begin{equation*}
	y_{n+1}\leq \frac{c4^{n}\|u\|_{2q,Q_T}^4}{A_rk}\{\vp\geq k_{n+1}\}|^{1-\frac{2}{q}+\frac{2}{N+2}}.
\end{equation*}
We are in a position to repeat our earlier argument.
If \eqref{hap} is not true, we can use the result in \eqref{hap1} to obtain the boundedness of $u$. The proof is rather standard. We shall omit here. The proof of \eqref{ub} is now complete.  
\end{proof}


\end{document}